\documentclass[11pt,UTF-8,reqno]{amsart}
\usepackage{enumerate}
\usepackage{mhequ}
\usepackage[margin=1.0in]{geometry}
\usepackage{amssymb,url,color, booktabs,nccmath}
\usepackage[dvipsnames]{xcolor}
\usepackage{mathrsfs}
\usepackage{enumitem}
\usepackage{graphicx}
\usepackage{mathtools}
\usepackage{tikz}
\usepackage{graphicx}
\usetikzlibrary{decorations.pathreplacing,decorations.markings}

\usepackage{color}
\usepackage[colorlinks=true]{hyperref}
\hypersetup{
    linkcolor=blue,          
    citecolor=red,        
    filecolor=blue,      
    urlcolor=cyan
}

\definecolor{darkergreen}{rgb}{0.0, 0.5, 0.0}

\numberwithin{equation}{section}

\newcommand{\be}{\begin{eqnarray}}
\newcommand{\ee}{\end{eqnarray}}
\newcommand{\ce}{\begin{eqnarray*}}
\newcommand{\de}{\end{eqnarray*}}
\newtheorem{theorem}{Theorem}[section]
\newtheorem{meta-theorem}{Meta-Theorem}[section]
\newtheorem{lemma}[theorem]{Lemma}
\newtheorem{remark}[theorem]{Remark}
\newtheorem{definition}[theorem]{Definition}
\newtheorem{proposition}[theorem]{Proposition}
\newtheorem{Examples}[theorem]{Example}
\newtheorem{corollary}[theorem]{Corollary}

\def\Re{{\mathrm{Re}}}

\def\u{\mathbf{u}}

\def\[{{\Big[}}
\def\]{{\Big]}}
\def\<{{\langle}}
\def\>{{\rangle}}
\def\({{\Big(}}
\def\){{\Big)}}

\def\bx{{\mathbf{x}}}
\def\tr{\mathrm {tr}}

\def\dif{{\mathord{{\rm d}}}}

\def\min{{\mathord{{\rm min}}}}

\def\={&\!\!=\!\!&}
\def\YMH{\textnormal{\small \textsc{ymh}}}

 \newcommand{\eqdef}{\stackrel{\mbox{\tiny def}}{=}}

\def\cI{{\mathcal I}}

\def\cP{{\mathcal P}}
\def\cQ{{\mathcal Q}}

\def\cS{{\mathcal S}}

\def\cV{{\mathcal V}}

\def\mB{{\mathbb B}}
\def\mC{{\mathbb C}}
\def\mD{{\mathbb D}}
\def\mE{{\mathbb E}}

\def\mG{{\mathbb G}}

\def\mL{{\mathbb L}}
\def\mM{{\mathbb M}}
\def\mN{{\mathbb N}}
\def\mO{{\mathbb O}}

\def\mR{{\mathbb R}}
\def\mS{{\mathbb S}}
\def\mT{{\mathbb T}}

\def\mX{{\mathbb X}}

\def\mZ{{\mathbb Z}}

\def\1{{\mathbf{1}}}

\def\E{\mathbf E}

\def\geq{\geqslant}
\def\leq{\leqslant}
\def\ge{\geqslant}
\def\le{\leqslant}

\def\ddots{\boldsymbol{:}}

\def\E{\mathbf{E}}
\def\CP{\mathcal{P}}
\def\Tr{\mathrm{Tr}}

\def\bx{{\mathbf{x}}}
\def\tr{\mathrm {tr}}

\def\dif{{\mathord{{\rm d}}}}

\def\min{{\mathord{{\rm min}}}}

\def\tr{{\rm Tr}}

\def\={&\!\!=\!\!&}
\def\bt{\begin{theorem}}
\def\et{\end{theorem}}
\def\bl{\begin{lemma}}
\def\el{\end{lemma}}
\def\br{\begin{remark}}
\def\er{\end{remark}}
\def\bx{\begin{Examples}}
\def\ex{\end{Examples}}
\def\bd{\begin{definition}}
\def\ed{\end{definition}}
\def\bp{\begin{proposition}}
\def\ep{\end{proposition}}
\def\bc{\begin{corollary}}
\def\ec{\end{corollary}}

\def\so{\mathfrak{so}}
\def\su{\mathfrak{su}}
\def\u{\mathfrak{u}}

\def\mfg{\mathfrak{g}}

\def\geq{\geqslant}
\def\leq{\leqslant}
\def\ge{\geqslant}
\def\le{\leqslant}

\def\R{\mathbb R}
\def\C{\mathbb C}

\def\N{\mathbb N}  
   
\def\<{\langle} \def\>{\rangle}

\def\${|\!|\!|}

\allowdisplaybreaks

\makeatletter
\def\section{\@startsection{section}{1}%
  \z@{1.7\linespacing\@plus\linespacing}{.5\linespacing}%
  {\normalfont\scshape\centering}}
\def\subsection{\@startsection{subsection}{2}%
  \z@{1\linespacing\@plus.7\linespacing}{-.5em}%
  {\normalfont\bfseries}}
\makeatother

\tikzset{
colorloop/.style={OliveGreen},
colorline/.style={blue},
}

\tikzset{
dot/.style={circle,fill,inner sep=1pt},
  on each segment/.style={
    decorate,
    decoration={
      show path construction,
      moveto code={},
      lineto code={
        \path [#1]
        (\tikzinputsegmentfirst) -- (\tikzinputsegmentlast);
      },
      curveto code={
        \path [#1] (\tikzinputsegmentfirst)
        .. controls
        (\tikzinputsegmentsupporta) and (\tikzinputsegmentsupportb)
        ..
        (\tikzinputsegmentlast);
      },
      closepath code={
        \path [#1]
        (\tikzinputsegmentfirst) -- (\tikzinputsegmentlast);
      },
    },
  },
  mid arrow/.style={postaction={decorate,decoration={
        markings,
        mark=at position .5 with {\arrow[#1]{stealth}}
      }}},
 midarrow/.style={postaction={on each segment={mid arrow=red}}},
}

\begin{document}

%


\date{\today}

\title{Makeenko--Migdal equations for lattice Yang--Mills--Higgs}

\author{Hao Shen}
\address[H. Shen]{Department of Mathematics, University of Wisconsin - Madison, USA}
\email{pkushenhao@gmail.com}

\author{Scott Andrew Smith}
\address[S. A. Smith]{Academy of Mathematics and Systems Sciences,Chinese Academy of Sciences, Beijing, China
}
\email{ssmith@amss.ac.cn}

\author{Rongchan Zhu}
\address[R. Zhu]{Department of Mathematics, Beijing Institute of Technology, Beijing, China 
}
\email{zhurongchan@126.com}

\dedicatory{(for the 95th birthday of Professor Leonard Gross)}

\maketitle

\begin{abstract}
We derive a form of master loop equations for the lattice Yang--Mills--Higgs theory with structure group $SO(N)$, $U(N)$ or $SU(N)$.  Compared to the pure Yang-Mills setting, several new operations arise.  In fact, to obtain a closed recursion we must broaden the class of observables to include open Wilson lines. Our approach is based on the conditional Langevin dynamic and yields a concise proof via It\^o's formula.
\end{abstract}

\setcounter{tocdepth}{2}
\tableofcontents

\section{Introduction}

In this article we derive 
the Makeenko--Migdal equations for 
the lattice Yang--Mills--Higgs model. 
These are also known as the master loop equations or Dyson--Schwinger equations.

The study of Makeenko--Migdal equations for the ``pure'' Yang--Mills model,
i.e. without coupling to additional fields such as a Higgs field,
goes back to the original paper Makeenko--Migdal in \cite{MM1979}
who heuristically proposed these identities as recursions of Wilson loop observables.
In two dimensions, Kazakov and Kostov \cite{kazakov1980non,kazakov1981wilson} 
and Gopakumar-Gross \cite{gopakumar1995mastering} clarified
that 
one side
of the Makeenko--Migdal  identity may be interpreted as the alternating sum of derivatives of the Wilson loop with respect to the areas.
L\'evy \cite{Levy11} (in 2011) was  the first to provide a rigorous proof of the 2D Makeenko–Migdal equations
and also introduced a more general form of the equations, and his proof is
based on Wilson loop expectation formulas developed in the seminal papers 
on 2D Yang--Mills by  Driver \cite{Driver89} and Gross--King--Sengupta \cite{GKS89}.
 
 Still in 2D, after \cite{Levy11}, alternative derivations were given by Dahlqvist \cite{MR3554890} and Driver--Hall--Kemp \cite{Driver17} (with three proofs)
and Driver \cite{MR3982691}.
\cite{MR3631396} shows that two  proofs in  \cite{Driver17}  also work  over  compact surfaces, based
on  works by Sengupta \cite{Sengupta92,Sengupta97,Sengupta1997}
(see \cite{Fine91,Witten91,Witten1992} for related works, and \cite{Levy03,Levy06,Levy10} for further development and generalizations 
of \cite{Sengupta97} in the geometric settings).
More recently,  on the torus $\mT^2$,
Chevyrev \cite{Chevyrev19YM} constructed a random connection 1-form
whose holonomies along axis paths 
coincide in law with the corresponding observables 
in \cite{Levy03},  see also \cite{CCHS2d,Chevyrev2023} which characterize
the Yang--Mills measure on a suitable quotient space of connection 1-forms on $\mT^2$ as the unique invariant measure of the Yang--Mills dynamic.

In the continuum, rigorous Makeenko--Migdal  type equations are unknown when $d\ge 3$.
However, on the lattice, the model is more amenable, and 
Chatterjee \cite{Cha} first derived a Makeenko--Migdal  equation
which holds in arbitrary dimensions.
This has been reproved, extended, or strengthened in subsequent works
 \cite{Jafar,SSZloop,OmarRon,CPS2023}.
More recently,  \cite{SSZ2d} showed that the lattice Makeenko--Migdal equations
converge to their continuum counterparts as in \cite{Levy11}. 

So far, all the above works are concerned with the pure Yang--Mills model.
In this paper we give the first derivation of the
Makeenko--Migdal equations in the context of  lattice Yang--Mills coupled with a Higgs field, for not only the loop observables, but also observables associated with ``open lines''. Indeed, Wilson loop observables do not form a closed system in the Higgs setting, as some of the operations that arise map Wilson loops into Wilson lines (and vice versa).  Hence, to obtain a closed system of equations one must consider both classes of observables simultaneously.

To introduce the model, 
let  $G$ be a compact Lie group with Lie algebra $\mfg$,
and let $M$ be a fixed finite dimensional manifold 
on which we have a $G$-action that preserves the metric on $M$.
We will focus on the following cases:

(1)  $G=SO(N)$, which acts on  $M\in \{\R^N,\mS^{N-1}\}$ by multiplication,
where $\mS^{N-1} =\{(x_1,\cdots,x_N)\in \R:|x_1|^2+\cdots+|x_N|^2=1\}$  is the unit sphere of dimensions $N-1$ in 
$\R^N$ endowed with the standard Euclidean  inner product.

(2) $G\in \{SU(N),U(N)\}$, which acts on $M\in \{\C^N,\mS^{2N-1}\}$ by multiplication,
where $\mS^{2N-1}=\{(z_1,\cdots,z_N)\in \C:|z_1|^2+\cdots+|z_N|^2=1\}$
is the unit sphere in $\C^N$ endowed with the  Hermitian inner product
$(u,v)=u^* v = \bar u^t v\in \C$ for $u,v\in \C^N$
and recall $|u|^2=u^*u\in \R$.

The matrix groups $G$ are  endowed with the Hilbert--Schmidt inner product.
For any matrix $A$ we write $A^*$ for its conjugate transpose,
and we view every vector as an $N\times 1$ matrix. 
Note that as a topological manifold $\mS^{2N-1}$ in case (2) 
is a sphere of dimensions $2N-1$. Below when writing the notation $\mS^{2N-1}$ it is always understood as the sphere in $\C^N$,
and we do not introduce heavier notation such as $\mS^{N-1}_{\mR}$, $\mS^{2N-1}_{\mC}$.


The lattice Yang--Mills--Higgs (YMH) model on a finite lattice $\Lambda$
is defined by the following action:
\begin{equ}[e:CS]
\mathcal S_{\YMH}(Q,\Phi) 
= 
\beta  \sum_{p\in \CP^+_\Lambda}
 \Re\,\Tr(Q_p)
+\kappa \sum_{e\in E^+_\Lambda }
\Re\, (\Phi_x^* Q_e \Phi_y  )
+ \sum_{z\in \Lambda} V(|\Phi_z|^2)
\end{equ}
where $\beta,\kappa$ are coupling constants,
and $V$ is a polynomial potential.  
The Euclidean quantum lattice YMH model is given by well-defined probability measures 
\begin{equ}[e:LYMH]
	\dif\mu^{\YMH}_{\Lambda}(Q,\Phi)
	= Z_{\Lambda}^{-1}
	\exp\Big(\mathcal S_{\YMH} (Q,\Phi) \Big)
	\prod_{e\in E^+_\Lambda} \dif Q_e
	\prod_{z\in \Lambda} \dif\Phi_z
\end{equ}
where $\dif Q_e$ is the Haar measure on the Lie group $G$, and $Z_{\Lambda}$ is a normalization factor.  To ensure the  measure is well-defined, we assume there exists some $R>0$ and $m> |\kappa|$ such that  $V(|\Phi|^{2}) \leq -m|\Phi|^{2}$ for $|\Phi| \geq R$. In fact, we will assume $V$ is a suitable polynomial, and its explicit form is given in Section \ref{sec:YMHSDEs} below. 
See also \cite{SZZYMH,Cha2024} for the background of the model.
\footnote{Our expression \eqref{e:CS} is the same as \cite{Cha2024}, whereas 
in \cite{SZZYMH} the second term in the action is the square of the discrete covariant
derivative  $Q_e \Phi_y - \Phi_x$. They are equivalent up to a re-definition of $V$ and a rescaling of $\kappa$ by $2$.}

We now specify the meaning of the notation $\Lambda,E_\Lambda^+,\CP_\Lambda^+$ appearing in \eqref{e:CS}.
\begin{enumerate}
\item
$\Lambda=\Lambda_{L}=\mZ^d\cap L\mT^d$
is a finite $d$ dimensional lattice
with side length $L>0$ and unit lattice spacing, and we will consider various functions on it with periodic boundary conditions.
\item
We say that a lattice edge $e=(u(e),v(e))$ of $\mZ^d$ is positively oriented if the beginning point $u(e)$
is smaller in lexographic order than the ending point $v(e)$.
Let $E^+$ (resp. $E^-$) be the set of positively (resp. negatively) oriented edges,
and  denote by $E_{\Lambda}^+$, $E_{\Lambda}^-$ the corresponding subsets
of edges with both beginning and ending points in $\Lambda$.  Define $E\eqdef E^+\cup E^-$. In \eqref{e:CS} we have written $e=(x,y)$.
\item

A {\it plaquette} $p=e_1e_2e_3e_4$ is a lattice loop enclosing a $1\times 1$ square,
which is a concatenation of four oriented edges $e_1,e_2,e_3,e_4$ (of course their cyclic order does not matter, see Section~\ref{sec:Term} for the precise definition of loops).
Let $\cP$ denote the set of plaquettes and let $\cP^+$ be the subset of 
counterclockwise plaquettes (namely the cross product $e_1\times e_2$ is positively oriented).
Also, let $\CP_{\Lambda}$ be the set of plaquettes whose vertices are all in $\Lambda$, and
 $\CP^+_{\Lambda}=\cP^+\cap \CP_{\Lambda}$.
\end{enumerate}

Moreover, in  \eqref{e:CS}, the discrete Yang--Mills field 
$Q=(Q_e : e\in E_{\Lambda}^+)$ is a collection of $G$-matrices.
The first term on the RHS of \eqref{e:CS} is the lattice Yang--Mills proposed by Wilson  \cite{Wilson1974}, where 
 $Q_p \eqdef Q_{e_1}Q_{e_2}Q_{e_3}Q_{e_4}$ for a plaquette $p=e_1e_2e_3e_4$.  Throughout the paper we define $Q_{e}\eqdef Q_{e^{-1}}^{-1}$ for $e \in E^{-}$, where $e^{-1}$ denotes the edge with orientation reversed.
The discrete Higgs field $\Phi=(\Phi_x : x\in \Lambda)$ is a collection of $M$-valued variables.
With all these definitions, the probability measure 
\eqref{e:LYMH} is defined on the configuration space defined by the product manifold
$$
\cQ_L\eqdef G^{E^+_\Lambda}\times M^\Lambda\;.
$$

It is well-known that under suitable scaling and recentering, 
\eqref{e:CS} converges the classical YMH action
as the lattice passes to the continuum
(c.f. \cite[Section~3]{Chatterjee18}, \cite{SZZYMH}.)
Note 
 that there are various other versions of the lattice Yang--Mills actions besides the above Wilson action,
for example Villain's action defined via heat kernels (c.f. \cite[Section~8]{Driver89}, \cite{Levy03,Levy06}),
or Manton's action \cite{Manton80} defined via the Riemannian metric on $G$; see also \cite[Section~2]{Chevyrev2023}.
We will not consider these choices in this paper.

For  any $g:\Lambda \to G$, we define gauge transformations by
\begin{equ}[e:gauge]
Q_e \mapsto g_x Q_e g_y^{-1},
\qquad
\Phi_x \mapsto g_x\Phi_x \;,
\qquad
(e=(x,y))\;.
\end{equ}
The action \eqref{e:CS}, and thus the measure \eqref{e:LYMH},
is invariant under such gauge transformations (c.f. \cite[Lemma~2.1]{SZZYMH}).
Due to gauge invariance of the model,
we are interested in the correlation functions of observables which are gauge invariant.
The class of gauge invariant observables we will consider 
are Wilson loop variables and Wilson line variables. 

A {\it path}  is defined to be a sequence of edges $e_1e_2\cdots e_n$ with $e_i\in E$ and $v(e_i)=u(e_{i+1})$ for $i=1,2,\cdots, n-1$. The path is called closed if $v(e_n)=u(e_1)$.
To define the observables, we  need the notion of 
{\it open lines} which are  paths with ``interior backtracks'' erased,
and {\it loops} which are closed paths modulo cyclic equivalence and also
with ``backtracks'' erased. Postponing the precise definitions of these terminologies to Section~\ref{sec:Term}, we define:

\begin{definition}[Wilson loops]
\label{def:Wloops}
Given a loop $ \ell = e_1 e_2 \cdots e_n$, the Wilson loop variable $W_\ell$
for the YMH field $(Q,\Phi)$ and the loop $\ell $
 is defined as 
$$
W_\ell = \Tr (Q_{e_1}Q_{e_2}\cdots Q_{e_n})\;.
$$
\end{definition}

\begin{definition}[Wilson lines]
\label{def:Wlines}
Given an open line $\ell = e_1 e_2 \cdots e_n$,
the Wilson line variable 
$W_\ell$
for the YMH field $(Q,\Phi)$ and the line $\ell$  is defined as 
$$
W_\ell = \Phi_{u(e_1)}^* Q_{e_1}Q_{e_2}\cdots Q_{e_n}\Phi_{v(e_n)}\;.
$$
Here $u(e),v(e)$ are the beginning and ending points of an edge $e$.
\end{definition}

It is easy to show that the Wilson loop and Wilson line variables
are invariant under gauge transformations \eqref{e:gauge} for any $g$, which is why they are interesting. 
Also, we note that the loop and line variables could be viewed
as ``macroscopic'' versions of the terms
in \eqref{e:CS}: $\tr(Q_p)$ is just a loop variable
where the loop is a single plaquette, 
and the other terms 
on the RHS of \eqref{e:CS} are just line variables 
where the line has one or zero edge.  Note that in the complex case, we are using here that $\text{Re}\text{Tr}Q_{p}=\frac{1}{2} \big (\text{Tr}Q_{p}+\text{Tr}Q_{p^{-1}} \big )$ and $\text{Re}(\Phi_{x}^{*}Q_{e}\Phi_{y} )=\frac{1}{2} \big (\Phi_{x}^{*}Q_{e}\Phi_{y}+\Phi_{y}^{*}Q_{e^{-1}}\Phi_{x} \big )$. 
Also note that $|\Phi_z|^4=(\Phi_z^* \Phi_z)^2$ can be viewed as two Wilson lines of zero length. Note that in Definition~\ref{def:Wlines} of Wilson line variables,
 we  allow the possibility that $u(e_1)=v(e_n)$, and in this situation the Wilson line variable is not a Wilson loop variable, due to the presence of $\Phi$. 
In fact, in this case, the open line $l  = e_1 e_2 \cdots e_n$
is {\it not} considered as a loop, since here $e_1$ is the first edge of $l$,
whereas a loop is defined modulo cyclic equivalence, see Section~\ref{sec:Term} below.
 
We also remark that for a loop or line $\ell$ and its reverse $\ell^{-1}$, 
the corresponding Wilson observables are equal in the real case,
but differ by a complex conjugate in the complex case.
To treat the real and complex cases uniformly,
we always treat loops and lines as directed, i.e. $\ell $ and its reverse $\ell^{-1}$
are two different loops / lines.

Loops and lines will be all called strings. We caution the reader that this differs from
the terminology ``string'' in earlier papers such as \cite{Cha}.
We will consider a collection of strings $s=(\ell_1,\cdots,\ell_k)$
(again see Section~\ref{sec:Term}) and write
\begin{equ}
W_s \eqdef\prod_{i=1}^{k}W_{\ell_{i}},
\qquad
\phi(s)\eqdef\E W_s.
\end{equ}
Our main result is the following Makeenko--Migdal equations. These equations state that the Wilson line and loop correlations $\phi(s)$
satisfy recursive relations. Namely, for a collection $s$ of strings,
$\phi(s)$ is expressed as a linear combination of $\phi(s')$
where $s'$ is obtained by applying certain operations on $s$.
In the following two tables we list these operations, heuristically 
describe what each operation does, define the coefficients
appearing in the linear combination, and defer their precise definitions to various subsections.
The first table lists operations occurring at an edge. 

\begin{center}
\begin{tabular}{c c c c}
\textbf{Operation $\mO$} & \textbf{Transformation}   & $q(\mO)$ & \textbf{Section}\\
\hline
Deformation 
& loop $\Rightarrow$ loop,\; line $\Rightarrow$ line 
& $\beta$
& \ref{sec:Deformation} 
\\
Breaking    
& loop $\Rightarrow$ line,\; line $\Rightarrow$ line+line 
& $\kappa$
& \ref{sec:Breaking} 
\\
Expansion by a plaquette  
& string $\Rightarrow$ string+loop 
& $\beta\nu$
& \ref{sec:Expansion} 
\\
Expansion by an edge  
& string $\Rightarrow$ string+line 
& $\kappa\nu$
& \ref{sec:Expansion} 
\\
Merger      
& loop+loop $\Rightarrow$ loop,\; line+loop $\Rightarrow$ line 
& $2\lambda$ or $2\mu$
& \ref{sec:Merger} 
\\
Switching   
& line+line $\Rightarrow$ line+line 
& $2\lambda$ or $2\mu$
& \ref{sec:Switching} 
\\
Splitting   
& loop $\Rightarrow$ loop+loop,\; line $\Rightarrow$ line+loop 
& $2\lambda$
&\ref{sec:Splitting}  
\\
Twisting 
& loop $\Rightarrow$ loop, \; line $\Rightarrow$ line 
& $2\mu$
&\ref{sec:Twisting}  
\\
\end{tabular}
\end{center}

Here, ``line  $\Rightarrow$ line'' indicates that a line is still a line after deformation,
and  ``line  $\Rightarrow$ line+line'' indicates that a line becomes a pair of lines after breaking, etc. The constants $\lambda,\mu,\nu$ are certain intrinsic constants depending on the choice of $G$, see \eqref{e:lambda-mu-nu}.
For mergers and switchings, they are further divided into two slightly different types of operations, with coefficients being either $2\lambda$ or $2\mu$,
and these are precisely defined in Section~\ref{sec:strings}.
We will write $\mO_e(s)$ for the set of string collections obtained from applying  this operation to a string collection $s$
at the edge $e$. Each $\mO_e(s)$ will further be divided into sets of  ``positive'' and ``negative''
versions of this operation, denoted as $\mO^+_e(s)$ and $\mO^-_e(s)$.

The following table 
 lists operations occurring at a lattice site.
 
 \begin{center}
\begin{tabular}{c c c c}
\textbf{Operation $\mO$} & \textbf{Transformation}   & $q(\mO)$ & \textbf{Section}\\
\hline
Extension   
& line $\Rightarrow$ line 
& $\kappa$
& \ref{sec:Extension} 
\\
Expansion  by an edge
&  line $\Rightarrow$ line+line 
& $-\kappa 1_{\mS}$
& \ref{sec:Expansion} 
\\
Expansion  by a null-line
& line $\Rightarrow$ line+ null-lines 
& $\vec{c} (1-1_{\mS})$
& \ref{sec:Expansion} 
\\
Gluing      
& line $\Rightarrow$ loop,\; line+line $\Rightarrow$ line 
& $2q$
&\ref{sec:Gluing} 
\\
$\R$-Gluing      
& line $\Rightarrow$ loop,\; line+line $\Rightarrow$ line 
& $2(2-q)$
&\ref{sec:Gluing} 
\\
\end{tabular}
\end{center}
Here $1_{\mS}$ is the indicator for $M$ being a sphere, see \eqref{e:1_S},
$q=1$ in the real case and $q=2$ in the complex case,
and $\vec{c} $ is a vector depending on $V$.
We will write $\mO_x(s)$ for the set of string collections obtained from applying  this operation to a string collection $s$
at the site $x$.

\begin{theorem}\label{theo:main}
Let $s$ be a collection of strings.

Fix an edge $e$.
One has
\begin{equs}
C_{G,s,e}\,\phi(s) =\sum_{\mO}  \mp q(\mO) \sum_{s' \in \mO^\pm_{e}(s)}\phi(s'),
\end{equs}
where the sum is over all operations in the first table,
and $C_{G,s,e}$ is a constant depending on $G,s,e$.

Fix  $x\in \Lambda$.
One has
\begin{equs}
C_{M,s,x}\,\phi(s) =\sum_{\mO}   q(\mO) \sum_{s' \in \mO_{x}(s)}\phi(s'),
\end{equs}
where the sum is over all operations in the second table,
and $C_{M,s,x}$ is a constant depending on $M,s,x$.
\end{theorem}

After all the operations are precisely defined, we will state the above theorem
more explicitly, see Theorem~\ref{theo:e-final} and Theorem~\ref{theo:x-final}.

\begin{remark}\label{rem:no-N}
In this article  $N$ is fixed. One could also consider large $N$ problems using the 
Makeenko--Migdal equations. In this case, typically $\beta,\kappa$
in \eqref{e:CS}
will depend on $N$.  
For instance the Yang--Mills term will be replaced by 
$N\beta  \Re \sum_{p} \Tr(Q_p)$,
which is the well-known t'Hooft scaling for the coupling $\beta$. 
In such settings, our Theorem~\ref{theo:main} then still holds with a simple change of variables $\beta,\kappa$. 
\end{remark}

Our main tool to prove Theorem~\ref{theo:main} is 
the Langevin dynamics, or stochastic quantization, 
as well as It\^o's formula. 
The dynamical approach was used to derive the Makeenko--Migdal equations
for the pure lattice Yang--Mills model without a Higgs field
\cite{SSZloop} (originally proved by \cite{Cha,Jafar}, see also \cite{CPS2023}).
Following the approach therein, we would consider 
the Langevin dynamics of  the measure \eqref{e:LYMH}
\begin{equ}[e:LSDE]
	\dif (Q,\Phi) = \nabla \mathcal S_\YMH (Q,\Phi) \dif t + \sqrt 2\dif \mathfrak B\;,
\end{equ}
where $\nabla$ is the gradient and $\mathfrak B$ is  the  Brownian motion on 
the product Riemannian manifold $\cQ_L=G^{E^+_\Lambda}\times M^\Lambda$. 
We refer to \cite[Section~2.1]{SZZ22} or \cite{SZZYMH} for detailed review 
on the geometry background.
In this paper we actually consider an improved approach, by studying the  
 Langevin dynamics of the {\it conditional} measure,
fixing all the variables except at one single edge $e$ or lattice site $x$.
Namely, a dynamic given by an stochastic differential equation (SDE for short) on $G$ or $M$, see \eqref{eq:YM1}.
We will give more explicit forms for this equation 
in Section~\ref{sec:YMHSDEs}.

In Section~\ref{sec:strings}, we precisely define 
the operations on strings. In Section~\ref{sec:Ito}, we prove our main
theorem.  There are multiple motivations to derive the 
 Makeenko--Migdal equations for the YMH model.
For lattice pure Yang--Mills,
such equations \cite{Cha,Jafar,SSZloop,OmarRon,CPS2023}
have been used to study (or have close connection with)
gauge-string duality, surface summations, large $N$ properties, area law,
and properties of the model in 2D, see  
\cite{Cha,ChatterjeeJafar,MR3861073,CPS2023,Cao2025area,borga2024surface}.
Hopefully our Makeenko--Migdal equations in Theorem~\ref{theo:main}
will also be useful to study these questions for the YMH model,
for example the question of surface sums where the surfaces have boundaries which arise from the Wilson line observables, or whether there are  transitions between area law and perimeter law depending on the parameters $N,d,\beta,\kappa$. 

Moreover, in the continuum 2D case,  the Makeenko--Migdal equations
for pure Yang--Mills \cite{Levy11,MR3554890,Driver17,MR3982691,PPSY2023}
play a significant role in the study of the model, and they can be derived as 
the continuum limit of the lattice Makeenko--Migdal equations, see \cite{SSZ2d}.
However, in continuum, these equations are not known once another field such as Higgs is coupled in the model, even in 2D. It would be interesting to see if our Theorem~\ref{theo:main} would give some hint on the formulation of such equations
for 2D YMH (note that the scaling limit approach in  \cite{SSZ2d} does not directly apply to YMH here since it relies on the special structures of 2D pure YM).
We leave these questions to future studies.

{\bf Acknowledgement.}
H.S. gratefully acknowledges financial support from an NSF grant (CAREER DMS-2044415).  S.S. and R.Z.  are grateful to the financial supports by National Key R\&D Program of China (No. 2022YFA1006300).
R.Z. is grateful to the financial supports of the NSFC (No. 12271030, No. 12426205) and the financial supports  by the Deutsche Forschungsgemeinschaft (DFG, German Research Foundation) – Project-ID 317210226--SFB 1283.

\section{Yang--Mills--Higgs SDEs}
\label{sec:YMHSDEs}

\subsection{Conditional SDEs}

Recall from \eqref{e:LYMH} that
\begin{equ}[e:LYMH-s2]
	\dif\mu^{\YMH}_{\Lambda}(Q,\Phi)
	= Z_{\Lambda}^{-1}
	\exp\Big(\mathcal S_{\YMH} (Q,\Phi) \Big)
	\prod_{e\in E^+_\Lambda} \dif  Q_e
	\prod_{z\in \Lambda} \dif\Phi_z \;.
\end{equ}
Fix  $e\in E_\Lambda^+$ and $x\in \Lambda$. Consider the regular conditional probability given $(Q_{\backslash e} , \Phi)\eqdef (Q\backslash Q_e,  \Phi) $ 
 \begin{equ}[e:LMHY-com]
 	\mu_{Q_{\backslash e},\Phi}(\dif Q_e)=\frac1{Z_{Q_{\backslash e},\Phi}}\exp(\cS_{\YMH}(Q,\Phi))\,\dif Q_e,
\end{equ}
 with 
 $$Z_{Q_{\backslash e},\Phi}=\int \exp(\cS_{\YMH}(Q,\Phi))\,\dif Q_e,$$
 and similarly, the  regular conditional probability given $(Q, \Phi_{\backslash x})\eqdef (Q,  \Phi\backslash \Phi_x) $ 
 \begin{equ}[e:LMHY-com-1]
 	\mu_{Q,\Phi_{\backslash x}}( \dif \Phi_x)
	=\frac1{Z_{Q,\Phi_{\backslash x}}}\exp(\cS_{\YMH}(Q,\Phi))\,\dif \Phi_x,
\end{equ}
 with 
 $$Z_{Q,\Phi_{\backslash x}}=\int \exp(\cS_{\YMH}(Q,\Phi))\,\dif \Phi_x.$$
 We can then use disintegration to write
 \begin{equ}[mea:dis]
 	 \mu^{\YMH}_\Lambda(\dif Q,\dif \Phi)= \nu_e(\dif Q_{\backslash e},\dif \Phi) \mu_{Q_{\backslash e},\Phi}(\dif Q_e)
= \nu_x(\dif Q,\dif \Phi_{\backslash x}) \mu_{Q,\Phi_{\backslash x}}(\dif \Phi_x)\;,
 \end{equ}
 with $\nu_e$ and $\nu_x$ being marginal laws of $(Q_{\backslash e},\Phi)$ and $(Q,\Phi_{\backslash x})$, respectively. 
We  consider the Langevin dynamics associated with the conditional probabilities in \eqref{e:LMHY-com} and \eqref{e:LMHY-com-1}, which are
the following SDEs on $G$ and $M$
\begin{equs}[eq:YM1]
	\dif Q_e &= \nabla_e \mathcal S_\YMH (Q,\Phi) \dif t + \sqrt 2\dif \mathfrak B_e\;,\\
		\dif \Phi_x &= \nabla_x \mathcal S_\YMH (Q,\Phi) \dif t + \sqrt 2\dif \mathfrak B_x\;.
\end{equs}

We first recall the definitions of the 
  gradients $\nabla_e$, $\nabla_x$  
  and  the  Brownian motions $\mathfrak B_e$ and $\mathfrak B_x$.  
  Given 
$(Q,\Phi)=\{(Q_e,\Phi_x ): e\in E_{\Lambda}^+,x\in \Lambda\}$,
let $\{v^i_e : 1\le i \le \dim(\mfg)\} $ and
$\{v^j_x : 1\le j \le \dim(M)\}$ be an orthogonal basis of the tangent space at $Q_e$ and $\Phi_x$ respectively.
For any function $f \in C^{\infty}(G)$, $g\in C^\infty(M)$ 
the gradient  in \eqref{eq:YM1} is defined as
\begin{equ}[e:def-grad] 
\nabla_{e}f \eqdef \sum_{i=1}^{\dim(\mfg)}(v_e^if)v_e^i,\qquad \nabla_{x}g\eqdef \sum_{j=1}^{\dim(M)}(v_x^jg)v_x^j,
\end{equ}
 where $vf$ is the usual notation in geometry for the derivative of $f$ along a tangent vector $v$.
 
Since $G$ is a Lie group,
every $X\in \mfg$ induces a right-invariant vector field 
on $G$, which is  just given by $XQ$
at each $Q\in G$ since $G$ is a matrix Lie group. 
The inner product on $\mfg$ induces an inner product
on the tangent space at every $Q\in G$
via the right multiplication on $G$.
Hence, for $X,Y\in \mfg$, 
the inner product of 
 $XQ,YQ \in T_{Q}G$ is given by $\Tr((XQ)(YQ)^*) = \Tr(XY^*)$.
This yields a bi-invariant Riemannian metric on $G$.
The (right-invariant) derivative of $f$ along $XQ$ can be calculated by
$ \frac{\dif}{\dif t}|_{t=0} f(e^{tX} Q)$.
We assume that our $\{v^i_e : 1\le i \le \dim(\mfg)\} $ is 
generated by an orthonormal basis of $\mfg$ in this way.

Also, in  \eqref{eq:YM1}, $\mathfrak B_e$ (resp. $\mathfrak B_x$) is a Brownian motion on $G$ (resp. on $M$).
Here we recall the following (c.f. \cite{SSZloop,SZZYMH}):
\begin{enumerate}
\item
Let $\mathfrak{B}$ and $B$  be the Brownian motions on a Lie group $G$ and its Lie algebra $\mfg$ respectively.
One has
$\E \big[  \<B(s),X \> \<B(t),Y \>  \big] = \min(s,t) \< X,Y \>$
for $X,Y \in \mfg$.
They are related by the following SDE
\begin{equ}[e:dB]
	\dif \mathfrak B = \dif B \circ \mathfrak B = \dif B\, \mathfrak B
	+ \frac{c_\mfg}{2} \mathfrak B \dif t,
\end{equ}
where $\circ$ is the Stratonovich product, and $\dif B\, \mathfrak B$ is in the It\^o sense.
Here the constant $c_\mfg$ is determined by
$ \sum_{\alpha} v_\alpha^2  =c_\mfg I_N$ where
$(v_\alpha)_{\alpha=1}^{\dim(\mfg)}$ is an orthonormal basis of $\mfg$, and
\begin{equ}[e:c_g]
	c_{\so(N)} =  -\frac12(N-1),
	\quad
	c_{\u(N)} =  -N,
	\quad
	c_{\su(N)} =  -\frac{N^2-1}{N} .
\end{equ}
For any adapted matrix-valued processes $M,N$ we have (\cite[(4.2)(4.3)]{SSZloop})
\begin{align}
		\dif B M \dif B &=\big (   \, \,\,\nu M  \, \,\,+\mu M^{t}-\lambda \tr M  \big ) \dif t \label{eq:M}. \\
		\text{Tr}(\dif B M)\text{Tr}(\dif B N)&=\Big ( \nu\text{Tr}(M)\text{Tr}(N)+\mu\text{Tr}(MN^{t})-\lambda \text{Tr}(MN) \Big ) \dif t \label{eq:MN}. 
	\end{align}
There are sometimes called ``magic formulas'' \footnote{In comparison to our prior work we make the change of notation $\lambda \to -\lambda$. Also note that $\nu$ denotes a constant whereas $\nu_x,\nu_e$ denote marginal laws in \eqref{mea:dis}.}, where
\begin{equs}[e:lambda-mu-nu]
\lambda=\frac12, \; \mu=\frac12, \; \nu=0 & \qquad (G=SO(N))
\\
\lambda=1, \; \mu=\nu=0 &\qquad (G=U(N))
\\
\lambda=1 ,\; \mu=0,\; \nu=\frac{1}{N} & \qquad (G=SU(N)).
\end{equs}
Due to these values,
when we apply \eqref{eq:M} \eqref{eq:MN}, we 
can assume that the $G$-matrices in the $\mu$ terms are elements of $SO(N)$,
and the $G$-matrices in the $\nu$ terms are elements of $SU(N)$.
\item
Let $\mathfrak{B}$ and $B$
denote the Brownian motions either on $\mS^{N-1}$  and $\mR^N$ respectively,
or, on $\mS^{2N-1}$  and $\mC^N$ respectively.
(In the latter case, $B$ is a collection of $N$ independent standard Brownian motions in $\mC$, each has real and imaginary parts as two independent standard Brownian motions in  $\mR$.)
In both cases, $\mathfrak{B}$ and $B$ are related by the following SDE
\begin{equ}\label{bm:sp}
\dif \mathfrak{B} 
=  \dif B - \mathfrak{B} \Re( \mathfrak{B}^*  \dif B) - \frac12 c_M \mathfrak{B} \dif t
\end{equ}
where we write $c_M$ for $M\in \{\mS^{N-1},\mS^{2N-1}\}$ for the It\^o constant given by
\begin{equ}[e:c_S]
c_{\mS^{N-1}}  = N-1,\qquad c_{\mS^{2N-1}}  = 2N-1\;.
\end{equ}
Here, as usual, $\Re$ does nothing in the real case.
Geometrically,
$u \Re(u^* v)$ is the projection of a vector $v$ into the direction of a vector $u$,
so $\dif B - \mathfrak{B} \Re( \mathfrak{B}^*  \dif B)$ is the projection of $\dif B$ to the tangent plane of the sphere at $ \mathfrak{B}$.
\end{enumerate}

For $e\in E^+$, we will write $B_e$ for the $\mfg$  valued Brownian motion associated with the $G$ valued Brownian motion $\mathfrak B_e$.

For $x\in \Lambda$,
we will write $B_x$ for the $\R^N$ or $\C^N$ valued Brownian motion 
and  $\mathfrak B_x$ the corresponding sphere valued Brownian motion.

\subsection{Explicit form of the SDEs}

Recall  from \eqref{e:CS} that $\mathcal S_{\YMH} = \cS_1+\cS_2+\cV$ where
\begin{equ}[e:SSV]
	\cS_1\eqdef \beta  \sum_{p\in \CP^+_\Lambda} \Re\,\Tr(Q_p),
	\qquad
	\cS_2\eqdef \kappa \sum_{e\in E^+_\Lambda }\Re (\Phi_x^*Q_e \Phi_y),
	\qquad
	\cV=\sum_{z\in \Lambda} V(|\Phi_z|^2)
\end{equ}
with $e=(x,y)$.

We now give an explicit form of the SDEs \eqref{eq:YM1}.
As in \cite{SSZloop}, for a plaquette $p$ and an oriented edge $e$,
 we write $p\succ e$ to indicate that one of the four oriented edges in $p$
 is $e$.
Note that for each edge $e$, there are $2(d-1)$ plaquettes in $\CP$
such that $p\succ e$.

For $e\in E_{\Lambda}^+$,  we have
\begin{equs}[SDE]
	\dif Q_e
	&=
	\nabla_e \cS_1 \,\dif t
	+ \nabla_e \cS_2 \,\dif t
	+\sqrt 2\, \dif \mathfrak B_e\;,
\end{equs}
where by It\^o's formula 
$$
\sqrt 2 \dif \mathfrak B_e = c_\mfg Q_e\dif t+\sqrt 2\dif B_eQ_e
$$
and as calculated  in  \cite{SSZloop} and \cite[Section~3]{SZZYMH}, 
\footnote{Note that our $\kappa$ in \eqref{e:CS} corresponds to $2\kappa$ in \cite{SZZYMH}.}
\begin{equs}[e:grad_e]
	\nabla_e\cS_1
	&=\begin{cases}
		\displaystyle 
		-\frac{\beta}2 \sum_{p\in \cP_\Lambda,p\succ e}(Q_p-Q_p^*)Q_e\;,
		&\qquad
		\mfg \in  \{\so(N),\u(N)\}\;,
		\\
		\displaystyle 
		- \frac{\beta}2 \sum_{p\in \cP_\Lambda ,p\succ e}
		\Big( (Q_p-{Q}_p^{*}) - \frac{1}{N}\tr(Q_p-{Q}_p^{*}) I_N\Big)   Q_e\;,
		&\qquad
		\mfg \in \su(N)\;.
	\end{cases}
	\\
	\nabla_e\cS_2
	&=\begin{cases}
		\displaystyle 
		\frac{\kappa}2 (\Phi_x\Phi_y^*-Q_e\Phi_y\Phi_x^* Q_e)\;, 
	&	\qquad	
			\mfg \in  \{\so(N),\u(N)\}\;,
		\\
		\displaystyle 
		\frac\kappa2\Big(\Phi_x\Phi_y^*Q_e^*-Q_e\Phi_y\Phi_x^* - \frac{1}{N}\tr(\Phi_x\Phi_y^*Q_e^*-Q_e\Phi_y\Phi_x^*) I_N\Big)   Q_e\;,
		&	\qquad
		\mfg \in \su(N)\;,
		\end{cases}
\end{equs}
for $e=(x,y)$.

For each $x\in \Lambda$,
we have
\begin{equs}[SDE1]
	\dif \Phi_x
	&=
	 \nabla_x \cS_2\, \dif t
	 + \nabla_x \cV \,\dif t
	+\sqrt 2\,\dif \mathfrak B_x\;,
\end{equs}
where for $M=\mR^N$ or $M=\mC^N$,
\begin{equ}
\sqrt 2\dif \mathfrak B_x
=\sqrt 2 \dif B_x	
\end{equ}
and 
\begin{equ}[e:grad_x]
\nabla_x \cS_2
=
\kappa \!\!\!\! 
\sum_{e=(x,y)\in E_\Lambda}  
	\!\!\!\! Q_e\Phi_{y} ,
\qquad\qquad
\nabla_x \cV
=
\sum_{k=0}^n c_k \Phi_x |\Phi_x|^k
\end{equ}
for some coefficients $c_0,\cdots,c_n$ with $c_n<0$ and $n=\deg(V)-1$.
For $M$ being the spheres, by It\^o's formula,
\begin{equ}[e:frakB2]
\sqrt 2\dif \mathfrak B_x
=
\sqrt 2(\dif B_x-\Phi_x\Re(\Phi_x^*\dif B_x)) -c_M\Phi_x\dif t,
\end{equ}
with $c_M$ as in \eqref{e:c_S} and again as in  \cite[Section~3]{SZZYMH}, 
\begin{equ}[e:grad_xS]
	\nabla_x \cS_2
=
	\kappa \sum_{e=(x,y)\in E_\Lambda}\Big(Q_e\Phi_{y}
	- \Re (\Phi_x^* Q_e \Phi_y)  \Phi_x\Big) ,
\qquad\qquad
\nabla_x \cV
=0 .
\end{equ}

As proved in \cite{SZZYMH}, the above SDEs \eqref{SDE} and \eqref{SDE1} are globally well-posed  in the state space $G$ and  $M$. In fact, it is even easier here since for the case that $(Q_{\backslash e} , \Phi)$ is fixed, the coefficients of the SDE \eqref{SDE} are Lipschitz. For fixed $(Q , \Phi_{\backslash x})$, we have dissipation from $c_n\Phi_x|\Phi_x|^n$ with $c_n<0$, which implies the global well-posedness of the SDE \eqref{SDE1} by standard arguments (see e.g. \cite[Chapter~3]{MR3410409}).  Also, the conditional probability measure \eqref{e:LMHY-com} is invariant under the SDE \eqref{SDE}, while \eqref{e:LMHY-com-1} is an invariant measure for the SDE \eqref{SDE1} 
(in the case $M=\R^N$ and $V(|\Phi|^2)=-m|\Phi|^2$ we need an additional assumption $m>|\kappa|$, see \cite{SZZYMH}; otherwise $c_n<0$ as above suffices).

Recall that in \eqref{e:def-grad}, $\nabla_e$ is only defined for $e\in E^+$.
Now for $e\in E^-$, we {\it define by hand}  $\nabla_e \cS_i$ ($i=1,2$) to be the same 
expressions on the RHS of \eqref{e:grad_e}. 
The motivation of this definition is that 
 for $e\in E^-$, the variable $Q_e=Q_{e^{-1}}^{-1}$ also satisfies a dynamic,
 and the following lemma shows that its dynamic behaves in the same way
``as if $e\in E^+$''. 
This will be convenient later since we will not have to discuss whether 
an edge in a string is in $E^+$ or $E^-$:
the drift will have the same form in either case, 
and for the quadratic variation of martingales we only care about if a pair of edges is the same or opposite.
(Also, just like \eqref{eq:M}\eqref{eq:MN}, the following identities may have independent interest so we put them in a lemma.)
\begin{lemma}\label{lem:DM}
For all $e\in E_\Lambda=E_\Lambda^+\cup E_\Lambda^-$, 
we have
$\dif Q_e = \mathcal D_e \,\dif t + \dif \mathcal M_e$ where

(1)
 the drift 
$\mathcal D_e= c_\mfg Q_e +\nabla_e \cS_1+\nabla_e \cS_2$;

(2)
the martingale $\mathcal M_e$ satisfies
\begin{equs}
\text{Tr}\big(\dif \mathcal M_e \,P\, \dif \mathcal M_e \,R \big)
&=
2\Big(
\nu \text{Tr} (Q_e PQ_e R)
+\mu \text{Tr} ( P^t  R)-\lambda\text{Tr}(Q_eP)\text{Tr}(Q_eR) \Big)\dif t	\label{e:martingale1}
\\
\text{Tr}\big(\dif \mathcal M_e \,P\, \dif \mathcal M_{e^{-1}} \,R \big)
&=
-2\Big(
\nu \text{Tr} (Q_e PQ_e^* R)
+\mu \text{Tr} (Q_e P^t Q_e^t R)-\lambda\text{Tr}(P)\text{Tr}(R) 
\Big)\dif t						\label{e:martingale2}
\\
\Tr \big(\dif \mathcal M_e \,P\big) 
\Tr\big(\dif \mathcal M_e \,R \big)
&=
2\Big(\nu \Tr (Q_e P)\Tr(Q_e R)
+\mu \text{Tr} ( P R^t)-\lambda\text{Tr}(Q_ePQ_eR) \Big)\dif t		\label{e:martingale3}
\\
\text{Tr}\big(\dif \mathcal M_e \,P\big) \Tr\big( \dif \mathcal M_{e^{-1}} \,R \big)
&=-2
\Big(
\nu \text{Tr} (Q_e P) \Tr( RQ_e^*)
+\mu \text{Tr} (Q_e P Q_e R^t)-\lambda\text{Tr}(PR) 
\Big)\dif t						\label{e:martingale4}
\end{equs}
for any adapted $N$-by-$N$ matrix-valued processes $P,R$,
where $\lambda,\mu,\nu$ are as in \eqref{e:lambda-mu-nu}.

\end{lemma}

\begin{proof}
(1)
For $e\in E_\Lambda^+$ the statement is already in \eqref{e:grad_e}.
Consider $e\in E_\Lambda^-$.
We have
$\dif Q_e  = \dif Q_{e^{-1}}^*$ so we need to show 
$$
\mathcal D_{e^{-1}}^{*}=(c_\mfg Q_{e^{-1}} +\nabla_{e^{-1}} \cS_1+\nabla_{e^{-1}} \cS_2)^*=\mathcal D_e
$$
where $\mathcal D_e$ is given in the statement of the lemma.
By definition, $c_\mfg  Q_{e^{-1}}^*=c_\mfg  Q_e$.
It was also shown in 
 \cite[(4.15)]{SSZloop} that 
$(\nabla_{e^{-1}} \mathcal{S}_1)^{*}
= \nabla_{e} \mathcal{S}_1$ where $\nabla_{e} \mathcal{S}_1$ was defined 
to be \eqref{e:grad_e}.
Finally,
 for $e=(x,y) \in E^{-}$  so that   $e^{-1}=(y,x) \in E^{+}$,
$$
(\nabla_{e^{-1}}\mathcal{S}_2)^{*}=
\frac{\kappa}2(
\Phi_y\Phi_x^*
-Q_{e^{-1}} \Phi_x\Phi_y^* Q_{e^{-1}})^*
=
\frac{\kappa}2(
\Phi_x\Phi_y^*
-Q_{e} \Phi_y\Phi_x^* Q_{e}),\qquad G=SO(U), U(N)
$$
which is precisely the RHS of \eqref{e:grad_e}. A similar calculation holds for $G=SU(N)$. 

(2) Note in advance that \eqref{eq:M} implies that for adapted matrices $M$,$N$ it holds
\begin{equation}
\tr(N\dif B M \dif B) =\big (  \nu \tr(NM)  \, \,\,+\mu \text{Tr}(N^{t}M )-  \lambda \tr N\tr M  \big ) \dif t \label{eq:M2}.
\end{equation}
For $e\in E_\Lambda^+$,
$\dif\mathcal M_e = \sqrt 2\dif B_eQ_e$.
For $e\in E_\Lambda^-$,
we have
$\dif Q_e  = \dif Q_{e^{-1}}^*$
so 
\begin{equ}[e:dMe]
\dif\mathcal M_e = \sqrt 2 Q_{e^{-1}}^* \dif B_{e^{-1}}^*
=-\sqrt 2Q_{e} \dif  B_{e^{-1}}.
\end{equ}
For  $e\in E_\Lambda^+$, 
\begin{equ}
\text{Tr}\big(\dif \mathcal M_e \,P\, \dif \mathcal M_e \,R \big)
=2 \text{Tr}\big(\dif B_eQ_e \,P\, \dif B_eQ_e \,R \big)
\stackrel{\eqref{eq:M2}}{=} 
\mbox{RHS of \eqref{e:martingale1}}
\end{equ}
and for  $e\in E_\Lambda^-$, the LHS is equal to
$2 \text{Tr}\big(Q_e\dif B_{e^{-1}} \,P\, Q_e \dif B_{e^{-1}}  \,R \big)$
which yields exactly the same result \eqref{e:martingale1}, since the result is invariant under  $Q_e P\mapsto PQ_e$ and $Q_e R\mapsto R Q_e$,
or directly verified by using \eqref{eq:M2}.

Moreover, 
for  $e\in E_\Lambda^+$,
\begin{equ}
\text{Tr}\big(\dif \mathcal M_e \,P\, \dif \mathcal M_{e^{-1}} \,R \big)
\stackrel{\eqref{e:dMe}}{=} 
- 2 \text{Tr}\big(\dif B_eQ_e \,P\, Q_e^* \dif B_e \,R \big)
\stackrel{\eqref{eq:M}}{=} 
 \mbox{RHS of \eqref{e:martingale2}}
\end{equ}
and for  $e\in E_\Lambda^-$, the LHS is equal to
$- 2 \text{Tr}\big(Q_e \dif B_{e^{-1}}  \,P\, \dif B_{e^{-1}} Q_{e^{-1}} \,R \big)$
which also gives  \eqref{e:martingale2}, since this is just the previous expression with $P,R$ swapped, or directly verified by using \eqref{eq:M}.

Next, for  $e\in E_\Lambda^+$, 
\begin{equ}
\Tr \big(\dif \mathcal M_e \,P\big) 
\Tr\big(\dif \mathcal M_e \,R \big)
=
 2 \Tr\big(\dif B_eQ_e \,P\big)\Tr\big(  \dif B_e Q_e \,R \big)
\stackrel{\eqref{eq:MN}}{=} 
 \mbox{RHS of \eqref{e:martingale3}}
\end{equ}
and for  $e\in E_\Lambda^-$, the LHS is equal to
$2 \text{Tr}\big(Q_e \dif B_{e^{-1}}  \,P\big) \Tr\big(Q_e \dif B_{e^{-1}}  \,R \big)$
which also gives  \eqref{e:martingale3},
since the result is invariant under  $Q_e P\mapsto PQ_e$ and $Q_e R\mapsto R Q_e$,
or directly verified by using \eqref{eq:MN}.

Finally, for  $e\in E_\Lambda^+$, 
$$
\text{Tr}\big(\dif \mathcal M_e \,P\big) \Tr\big( \dif \mathcal M_{e^{-1}} \,R \big)
\stackrel{\eqref{e:dMe}}{=} 
- 2 \text{Tr}\big(\dif B_eQ_e \,P\big) \Tr\big(Q_e^* \dif B_e \,R \big)
\stackrel{\eqref{eq:MN}}{=} 
 \mbox{RHS of \eqref{e:martingale4}}
$$
and for  $e\in E_\Lambda^-$, the LHS 
 also gives  \eqref{e:martingale4},
since it is just the previous equation with the roles of $P,R$ swapped.
\end{proof}
In the sequel, we introduce the indicator notation for $M$ being a sphere
\begin{equ}[e:1_S]
\1_{\mS}=\1_{M=\mS^{N-1}}+\1_{M=\mS^{2N-1}}.
\end{equ}
In parallel with Lemma~\ref{lem:DM} we have: 

\begin{lemma}\label{lem:DMx}
For all $x\in \Lambda$, 
we have
$\dif \Phi_x = \mathcal D_x \,\dif t + \dif \mathcal M_x$ where

(1)
 the drift 
$\mathcal D_x= -c_M \Phi_x +\nabla_x \cV+\nabla_x \cS_2$ where $c_M$ is as in \eqref{e:c_S}
and $c_M=0$ if $M\in \{\mR^N,\mC^N\}$;

(2)
the martingale $\mathcal M_x$ satisfies 
\begin{equ}[e:martin-x1]
\dif \mathcal M_x^* P \dif \mathcal M_x  
= \Big(
2q\tr (P) - \1_{\mS} \,2  \Phi_x^* P  \Phi_x 
\Big)\dif t 
\end{equ}
and
\begin{equs}[e:martin-x2]
 \dif \mathcal M_x R^* \dif \mathcal M_x 
 &= \Big(
 2(2-q) R -\1_{\mS} 2 \Phi_x R^* \Phi_x \Big) \dif t 
 \\
 \dif \mathcal M_x^* R \dif \mathcal M_x^* 
& = \Big(
2(2-q) R^*  - \1_{\mS} 2 \Phi_x^* R \Phi_x^*\Big)\dif t
\end{equs}
for any adapted $N$-by-$N$ matrix-valued process $P$ and $N$-by-$1$ matrix-valued (i.e. vector-valued) $R$,
where $q=1$ in the real case and $q=2$ in the complex case.
\end{lemma}

\begin{proof}
(1)
This follows from \eqref{SDE1} and \eqref{e:frakB2}.

(2) We drop the subscript $x$ to simplify the notion, and write $B^i$, $P^{ik}$ etc. where $i,k\in\{1,\cdots,N\}$ for coordinates of vectors and matrices.  Note that
\begin{equs}
\dif \mathcal M^{*}P \dif \mathcal M
&=
2\sum_{i,k}
\Big(\dif \overline{B^i}
	-\1_{\mS}\overline{\Phi^i} \sum_j\Re(\overline{\Phi^j}\dif B^j)\Big)
	P^{ik}
\Big(\dif {B^k}- \1_{\mS}\Phi^k \sum_l \Re(\overline{\Phi^l}\dif B^l)\Big)
\\
&= 2\sum_i \dif \overline{B^i} P^{ii} \dif B^i
+ \1_{\mS}2\sum_{i,j,k} \overline{\Phi^i} \Re(\overline{\Phi^j}\dif B^j) P^{ik}
\Phi^k \Re(\overline{\Phi^j}\dif B^j)
\\
&\quad
-\1_{\mS}2\sum_{i,k} \dif \overline{B^i} P^{ik} \Phi^k \Re(\overline{\Phi^i} \dif B^i) 
-\1_{\mS}2\sum_{i,k}  \overline{\Phi^i} \Re(\overline{\Phi^k} \dif B^k) P^{ik} \dif B^k
\\&
=2q \tr(P) -\1_{\mS}2 \Phi^*P \Phi.
\end{equs}
Here recall again that the conjugate $\overline{(\cdot)}$ does nothing in the real case.
From the first line to the second line we used independence of Brownian coordinates.
For each $\Re(\cdots)$ we write for instance 
$ \Re(\overline{\Phi^k} \dif B^k) = \frac12  (\overline{\Phi^k} \dif B^k+ \Phi^k \dif \overline{B^k})$
and then we used $\overline{\dif B^i} \dif B^i = q\dif t$.  Recall
that in the complex case we have $\dif B^i \dif B^i = \dif \overline{ B^i} \dif \overline{B^i} = 0$.
For the term with $\sum_{i,j,k}$
we used $\sum_{j} \overline{\Phi^j} \Phi^j=1$.

Moreover,  we prove the second identity in \eqref{e:martin-x2} and the first one follows follows from the second identity by writing
\begin{equation}
	\dif \mathcal{M}_{x}R^{*} \dif \mathcal{M}_{x}=(\dif \mathcal{M}_{x}^{*}R \dif \mathcal{M}_{x}^{*} )^{*} \nonumber.
\end{equation}
	For the second identity, 
\begin{align*}
	(\dif \mathcal M_{x}^{*}R\dif \mathcal M_{x}^{*} )^{k}=2\sum_{j}(\overline{\dif B^{j}}-\overline{\Phi^{j}}\sum_{\ell}\text{Re}(\overline{\Phi^{\ell}}\dif B^{\ell} ) )R^{j}(\overline{\dif B^{k}}-\overline{\Phi^{k}}\sum_{m}\text{Re}(\overline{\Phi^{m}}\dif B^{m} ) )
\end{align*}
One term gives
\begin{equation*}
	\sum_{j}\overline{\dif B^{j}}R^{j}\overline{\dif B^{k}}=
	\begin{cases}
		R^{k} \qquad \text{real case}\\
		0 \qquad \text{complex case}
	\end{cases}
\end{equation*}
In the sphere case, two of the terms give
\begin{equation*}
	-2\sum_{j,m}\overline{\dif B^{j}}\overline{\Phi^{k}}\text{Re}(\overline{\Phi^{m}}\dif B^{m} )R^{j}=q \sum_{j} \overline{\Phi^{k}}\overline{\Phi^{j}}R^{j}=q(\Phi_{x}^{*}R \Phi_{x}^{*})^{k}
\end{equation*}
can be calculated in the same way as above
which gives \eqref{e:martin-x2}.
\end{proof}

\begin{remark}\label{rem:2}
Our ``dynamical proof'' will  follow similar idea as in \cite{SSZloop},
but with a few differences. First, we use the dynamic for the conditional measure here,
which will yield a stronger version of the loop equation.
Second, we use a dynamic of the form $\dif X=\nabla\cS \dif t + \sqrt 2 \dif \mathfrak{B}$
here whereas \cite{SSZloop} uses a dynamic of the form $\dif X=\frac12 \nabla\cS \dif t +  \dif \mathfrak{B}$; this is just a matter of convenience which may affect some intermediate calculations in the proof but the final results will be the same.

Another difference is that we do not introduce
extra factors of $N$ in \eqref{def:Wloops} and $\phi(s)$ as in \cite{SSZloop}  (see also Remark~\ref{rem:no-N}) so when citing certain results from  \cite{SSZloop}  we will take account of such differences.
\end{remark}

\section{Strings on the lattice}
\label{sec:strings}

We introduce the strings with their basic operations  in this section,
and define a discrete string theory which contains 
open and closed strings.
We start by defining some terminologies which already show up  in Def.~\ref{def:Wloops} and Def.~\ref{def:Wlines} more precisely.

\subsection{Loops and lines}
\label{sec:Term}

\subsubsection{Paths, closed paths, cycles}

A {\it path} $\rho$ in the lattice  is defined to be a sequence of edges $e_1,e_2,\cdots,e_n$ 
such that $v(e_i) = u(e_{i+1})$ for $i =1,2,\cdots,n-1$.
We write $\rho = e_1e_2 \cdots e_n$ and say that $\rho$ has length $n$, i.e. $|\rho|=n$. 
We call $u(e_1)$ the beginning point, and $v(e_n)$ the ending point.
Both are called endpoints.

The path $\rho$ will be called a {\it closed path} if $v(e_n) = u(e_1)$. 
We also define a ``null-path'' to be a path with no edges.  We define the inverse of the path $\rho$ as $\rho^{-1} =e_n^{-1} e_{n-1}^{-1}\cdots e_1^{-1}$.
For another path  $\rho'= e_1' \cdots e_m' $  such that $v(e_n) = u(e_1')$, then the concatenated path $\rho\rho'$ is defined as
 $\rho\rho'  =e_1e_2 \cdots e_n  e_1' \cdots e_m'$.
 
If  $\rho = e_1e_2 \cdots e_n$  is a closed path, we will say that another closed path $\rho'$ is 
{\it cyclically equivalent} to $\rho$, and write $\rho \sim \rho'$ if 
$$
\rho' = e_i e_{i+1} \cdots e_n e_1 e_2 \cdots e_{i-1}
$$
for some $2\le i \le n$. The equivalence classes are called {\it cycles}.
In other words  the closed path $\rho = e_1e_2 \cdots e_n$  has a notion of the beginning point $u(e_1)$ and the ending point $v(e_n)$ (which are the same point), whereas a cycle does not have such a notion.

\subsubsection{Backtracking}

We  now define the notion of backtracking as  \cite[Section~2]{Cha}.
We say that a path $\rho= e_1e_2\cdots e_n$
 has a ``backtrack'' at location $i$, where $1\le i\le n-1$,
if 
$e_{i+1}=e_i^{-1}$.
If $\rho$ is a closed path,
then we will say that it has a backtrack
at location $n$ if $e_1 = e^{-1}_n$. 
A backtrack that occurs at some
location $i\le n-1$ will be called an ``interior backtrack'', and a backtrack at location $n$ will be called a ``terminal backtrack''.

\begin{figure}[h]
\begin{tikzpicture}[baseline=15]
\draw[thick,midarrow] (0, .04) -- (1, .04) -- (1,1) -- (-1,1) -- (-1,-1) -- (1,-1) -- (1, -.04) -- (0, -.04);
\node at (0.5,0.25) {$e_1$};
\node at (1.25,0.5) {$e_2$};
\node at (1.25,-0.5) {$e_9$};
\node at (0.5,-0.25) {$e_{10}$};
\end{tikzpicture}
\caption{A closed path $e_1e_2\cdots e_{10}$ with a terminal backtrack $e_1=e_{10}^{-1}$.}\label{fig:path-term-bt}
\end{figure}
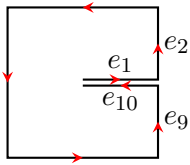

As in  \cite[Section~2]{Cha} given a path one can keep erasing backtracks (resp. interior backtracks) to obtain a  path without backtracks (resp. interior backtracks). Thanks to Lemma~\ref{lem:well-def}, 
we can  keep erasing backtracks in a cycle  to obtain a cycle  without backtracks.
The results do not depend on  the choice of the sequence of backtrack erasures.

We call the  paths without interior backtracks  {\it open lines} or simply {\it lines}, and the cycles without backtracks {\it loops}.

As in  \cite{Cha} we call loops with no edges null-loops.
Additionally, we add into the set of lines so-called null-lines,
denoted by $\ddots_x$ where $x$ is a lattice site.
One should think of it as having no edges and is a line from $x$ to $x$.

The following diagram illustrates these terminologies:
\begin{equs}[e:lines-loops]
{}&\quad\;
\{\mbox{closed paths}\}\subset  \{ \mbox{paths}\}
\qquad\qquad
 \{\mbox{cycles}=\mbox{closed paths} / \mbox{cyclic equivalence} \}
\\
& \mbox{erase interior backtracks}
\bigg\downarrow 
\qquad\qquad\qquad\qquad
\bigg\downarrow 
\mbox{erase all backtracks}
\\
&\qquad \mbox{add null-lines}\to 
\{\mbox{lines}\}
\qquad\qquad\qquad
\{\mbox{loops}\}
\end{equs}

We will draw  {\color{darkergreen}  loops in green},
and draw  {\color{blue} lines in blue} with two dots at the  endpoints. 
Loops and lines will be all called strings. Remark that this differs from
the terminology ``string'' in earlier papers such as \cite{Cha,SSZloop}.

According to diagram \eqref{e:lines-loops}, 
given the closed path in Fig.~\ref{fig:path-term-bt},
we can erase {\it all} backtracks (as in the right half of diagram \eqref{e:lines-loops}),
which creates a loop;
we can also only erase {\it interior} backtracks 
(as in the left half of diagram \eqref{e:lines-loops}, but for this example
in Fig.~\ref{fig:path-term-bt}
there is no interior backtrack so nothing to erase),
which creates a line.
These are shown in Fig.~\ref{fig:line-loop-bt}.

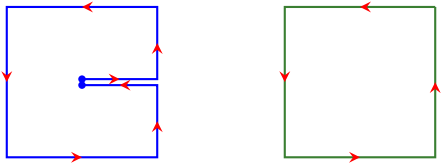
\begin{figure}[h]
\begin{tikzpicture}[baseline=15]
\draw[thick,colorline,midarrow] (0, .04) -- (1, .04) -- (1,1) -- (-1,1) -- (-1,-1) -- (1,-1) -- (1, -.04) -- (0, -.04);
\node[dot,colorline] at (0, .04) {};\node[dot,colorline] at (0, -.04) {};
\end{tikzpicture}
\qquad\qquad
\begin{tikzpicture}[baseline=15]
\draw[thick,colorloop,midarrow] (1,1) -- (-1,1) -- (-1,-1) -- (1,-1) -- (1,1);
\end{tikzpicture} 
\caption{A line and a loop created from the path in Fig.~\ref{fig:path-term-bt}.} \label{fig:line-loop-bt}
\end{figure}

\begin{remark}
We will often write a loop or line as a concatenation of edges or paths, 
for example, a line $\ell=aebec$ where $a,b,c$ are paths (without interior backtracks).
Here, when we write this way, $a$ and $c$ are allowed to be null paths, so $\ell$ may start or end with $e$. Also, note that where 
null-paths and null-loops are essentially ``nothing'', a null-line does have ``something'' which is a site $x$ and it will correspond to the nontrivial observable $\Phi_x^* \Phi_x$.
\end{remark}

\bl
\label{lem:well-def}
(1) Given a path, after  successively erasing interior backtracks, its two endpoints are not changed.

(2) Given a closed path, after  successively erasing interior backtracks, it is still a closed path.

(3) Let $\rho_1$ and $\rho_2$ be cyclically equivalent closed paths (namely they define the same cycle), 
and let $\rho_1'$ and $\rho_2'$ be two nonbacktracking closed paths obtained by successively erasing backtracks starting from 
$\rho_1$ and $\rho_2$ respectively. Then $\rho_1'$ and $\rho_2'$ are cyclically equivalent. 
\el

\begin{proof}
(1)(2) are obvious. (3) follows from \cite[Lemma~2.1]{Cha}.
\end{proof}

Both loops and lines will be called {\it strings}.
We write $\mL_0$ for the set of all the loops, $\mL_1$ for the set of all the open lines, and 
$$
\mL=\mL_0 \sqcup \mL_1.
$$
We have disjoint union here because loops are always cyclic equivalence classes whereas lines always have endpoints. 

As in \cite{Cha} we use $[\ell]$ to denote the nonbacktracking core of the string $\ell$. 
(Namely, thanks to Lemma~\ref{lem:well-def}, we can define a ``nonbacktracking
core'' which is simply the unique loop or line obtained by successive backtrack
erasures until there are no more backtracks.)
We will need to deal with a collection of strings. Suppose
$(\ell_1,\dots,\ell_n)$ and $(\ell_1',\dots,\ell_m')$ are two finite sequences of strings. These sequences are called equivalent if one can be obtained from the other by deleting or inserting null loops at arbitrary locations.  
The resulting  equivalence class is called a collection of strings.

\subsection{Operations}
\label{sec:Operations}

We define $4$ operations  called 
deformation, breaking,  extension and expansion
which will turn out to correspond to the drift terms of the YMH SDE,
as well as $5$ operations called merger, switching, splitting, twisting, and gluing
 which will turn out to correspond to the quadratic variations of the noise terms of the YMH SDE.
Recall that above Theorem~\ref{theo:main} 
we outlined informally what these operations will do on the strings.

Note that all these deformations that we will define below have 
clear geometric intuition, and, they will precisely arise when we apply It\^o's formula to the Wilson loop and line observables, which is an interesting feature of the model.

\subsubsection{Splitting}
\label{sec:Splitting}

Given a loop $\ell_0$, a line $\ell_1$,
and two  locations $x$ and $y$ in $\ell_i$ ($i\in \{0,1\}$),
we define the positive  and negative splittings of $\ell_i$ at locations $x$ and $y$
$$
\times_{x,y} \ell_i
$$
as follows.
If $\ell_i=aebec$ (where $a, b, c$ are paths and $e$ is an edge),  	
define the positive splitting
\begin{equ}[e:splitting1]
\times_{x,y} \ell_i = ( [aec] , [be])\in \mL_i\times \mL_0 \qquad i\in \{0,1\}.
\end{equ}
If $\ell_i=aebe^{-1}c$, 	
define negative splitting
\begin{equ}[e:splitting2]
\times_{x,y} \ell_i = ([ac], [b])\in \mL_i\times \mL_0 \qquad i\in \{0,1\}.
\end{equ}
We sometimes denote by 
$\times^1_{x,y} \ell_i$ and $\times^2_{x,y} \ell_i$  the
pair of strings obtained by splitting.

\begin{figure}[h]
\begin{tikzpicture}[baseline=15]
\draw[thick,colorline,midarrow,thick] (4,2)  -- (2.07,2) -- (2.07,0.07) -- (1,0.07) -- (1,-1) -- (0,-1) --(0,2) -- (1,2) -- (1,1) -- (2,1) -- (2,0) -- (3,0) -- (3,-1) -- (4,-1);
\node[dot,colorline] at (4,2) {};  \node[dot,colorline] at (4,-1) {};
\node at (1.8,0.5) {$e$};
\end{tikzpicture}
\qquad
\begin{tikzpicture}[baseline=15]
\draw[->,dashed,thick] (0,0) to (1,0);
\end{tikzpicture}
\qquad
\begin{tikzpicture}[baseline=15]
\draw[thick,colorloop,midarrow] (0,2) -- (1,2) -- (1,1) -- (2,1) -- (2,0) -- (1,0) -- (1,-1) -- (0,-1) -- (0,2);
\node at (1.8,0.5) {$e$};
\end{tikzpicture}
\begin{tikzpicture}[baseline=15]
\draw[thick,colorline,midarrow] (2,2)  -- (0,2) -- (0,0) -- (1,0) -- (1,-1) -- (2,-1);
\node[dot,colorline] at (2,2) {}; \node[dot,colorline] at (2,-1) {};
\node at (0.2,0.5) {$e$};
\end{tikzpicture}
\caption{Positive splitting of a line $\ell_i=aebec\in \mL_1$ as illustration of \eqref{e:splitting1}.}
\label{Positive splitting}
\end{figure}
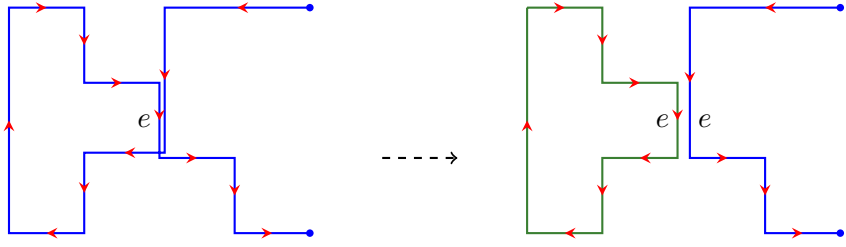

\subsubsection{Expansion}
\label{sec:Expansion}

We have four types of expansions:
(1) expansion at an edge $e$ by a plaquette-loop 
(which already existed in \cite{Jafar,SSZloop} in the $SU(N)$ loop equation),
(2) expansion at an edge $e$ by an edge-line, 
(3) expansion at a ``Higgs vertex'' (i.e. an endpoint of an open line) by an edge-line,
and, (4) expansion at a ``Higgs vertex'' by a null-line.  

For the first type (expansion at an edge $e$ by a plaquette $p$), as in \cite{SSZloop}, let $\ell\in \mL$ contain an edge $e$ at location $x$. 
A positive expansion of $\ell$ at location $x$ by a plaquette $p$ passing through 
$e^{-1}$ replaces $\ell$ with the pair $(\ell,p)$ of a string and a loop. 
A negative expansion of $\ell$ at location $x$ by a plaquette $p$ passing through
 $e$ replaces $\ell$ with the pair  $(\ell, p)$ of a string and a loop.

\begin{figure}[h]
\begin{tikzpicture}[baseline=15]
\draw[thick,colorline,midarrow] (1,0)-- (0,0) -- (0,1) -- (3,1) -- (3,0)  -- (2,0); 
\node[dot,colorline] at (1,0) {}; \node[dot,colorline] at (2,0) {};
\node at (2.8,0.5) {$e$};
\end{tikzpicture}
\qquad
\begin{tikzpicture}[baseline=5]
\draw[->,dashed,thick] (0,0) to (1,0);
\end{tikzpicture}
\qquad
\begin{tikzpicture}[baseline=15]
\draw[thick,colorline,midarrow] (1,0)-- (0,0) -- (0,1) -- (3,1) -- (3,0)  -- (2,0); 
\node[dot,colorline] at (1,0) {}; \node[dot,colorline] at (2,0) {};
\node at (2.8,0.5) {$e$};
\draw[thick,colorloop,midarrow] (3.07,0) -- (3.07,1) -- (4,1) -- (4,0)  -- (3.07,0); 
\node at (3.5,0.5) {$p$};
\end{tikzpicture}
\caption{Positive expansion of a line by $p$.}
\end{figure}
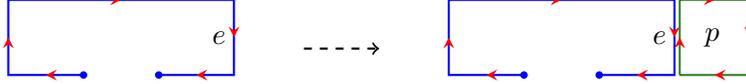

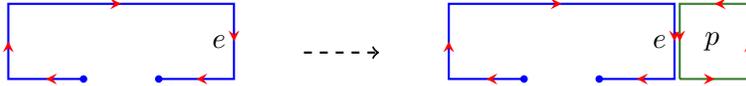
\begin{figure}[h]
\begin{tikzpicture}[baseline=15]
\draw[thick,colorline,midarrow] (1,0)-- (0,0) -- (0,1) -- (3,1) -- (3,0)  -- (2,0); 
\node[dot,colorline] at (1,0) {}; \node[dot,colorline] at (2,0) {};
\node at (2.8,0.5) {$e$};
\end{tikzpicture}
\qquad
\begin{tikzpicture}[baseline=5]
\draw[->,dashed,thick] (0,0) to (1,0);
\end{tikzpicture}
\qquad
\begin{tikzpicture}[baseline=15]
\draw[thick,colorline,midarrow] (1,0)-- (0,0) -- (0,1) -- (3,1) -- (3,0)  -- (2,0); 
\node[dot,colorline] at (1,0) {}; \node[dot,colorline] at (2,0) {};
\node at (2.8,0.5) {$e$};
\draw[thick,colorloop,midarrow] (3.07,0) -- (4,0) -- (4,1)  -- (3.07,1)   -- (3.07,0); 
\node at (3.5,0.5) {$p$};
\end{tikzpicture}
\caption{Negative expansion of a line by $p$.}
\end{figure}

For the second type (expansion at an edge $e$ by an edge), 
let $\ell\in \mL$ contain an edge $e$ at location $x$. 
A positive expansion of $\ell$ at location $x$ by 
$e^{-1}$ replaces $\ell$ with the pair  $(\ell,e^{-1})$ of a string and a line. 
A negative expansion of $\ell$ at location $x$ by 
$e$ replaces $\ell$ with the pair  $(\ell, e)$ of a string and a line.

	\begin{figure}[h]
\begin{tikzpicture}[baseline=15]
\draw[thick,colorline,midarrow] (1,0)-- (0,0) -- (0,1) -- (3,1) -- (3,0)  -- (2,0); 
\node[dot,colorline] at (1,0) {}; \node[dot,colorline] at (2,0) {};
\node at (2.8,0.5) {$e$};
\end{tikzpicture}
\qquad
\begin{tikzpicture}[baseline=5]
\draw[->,dashed,thick] (0,0) to (1,0);
\end{tikzpicture}
\qquad
\begin{tikzpicture}[baseline=15]
\draw[thick,colorline,midarrow] (1,0)-- (0,0) -- (0,1) -- (3,1) -- (3,0)  -- (2,0); 
\node[dot,colorline] at (1,0) {}; \node[dot,colorline] at (2,0) {};
\node at (2.8,0.5) {$e$};
\draw[thick,colorline,midarrow] (3.1,0) -- (3.1,1); 
\node[dot,colorline] at (3.1,0) {}; \node[dot,colorline] at (3.1,1) {};
\end{tikzpicture}
\caption{Positive expansion of a line by $e^{-1}$.}
\end{figure}
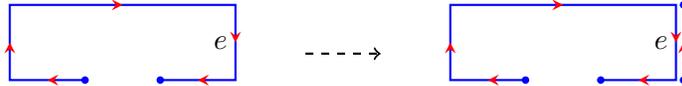

	\begin{figure}[h]
\begin{tikzpicture}[baseline=15]
\draw[thick,colorline,midarrow] (1,0)-- (0,0) -- (0,1) -- (3,1) -- (3,0)  -- (2,0); 
\node[dot,colorline] at (1,0) {}; \node[dot,colorline] at (2,0) {};
\node at (2.8,0.5) {$e$};
\end{tikzpicture}
\qquad
\begin{tikzpicture}[baseline=5]
\draw[->,dashed,thick] (0,0) to (1,0);
\end{tikzpicture}
\qquad
\begin{tikzpicture}[baseline=15]
\draw[thick,colorline,midarrow] (1,0)-- (0,0) -- (0,1) -- (3,1) -- (3,0)  -- (2,0); 
\node[dot,colorline] at (1,0) {}; \node[dot,colorline] at (2,0) {};
\node at (2.8,0.5) {$e$};
\draw[thick,colorline,midarrow] (3.1,1) -- (3.1,0); 
\node[dot,colorline] at (3.1,0) {}; \node[dot,colorline] at (3.1,1) {};
\end{tikzpicture}
\caption{Negative expansion of a line by $e$.}
\end{figure}
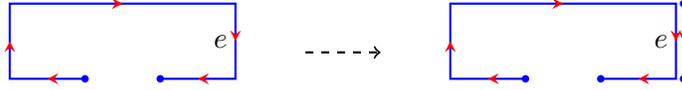

For the third type  (expansion at the end of a line by an edge-line), 
given a line $\ell_1$ and an edge $e$ where $v(\ell_1)=u(e)$ or $v(\ell_1)=v(e)$, 
we define the  expansion of $\ell_1$ by $e$ as
\begin{equ}[e:expansion1e]
 (\ell_1, e) \in \mL_1\times \mL_1.
\end{equ}
Given a line $\ell_1$ and an edge $e$ where $u(\ell_1)=v(e)$ or $u(\ell_1)=u(e)$, 
we define the  expansion of $\ell_1$ by $e$ as
\begin{equ}[e:expansion2e]
( e, \ell_1)\in \mL_1\times \mL_1.
\end{equ}

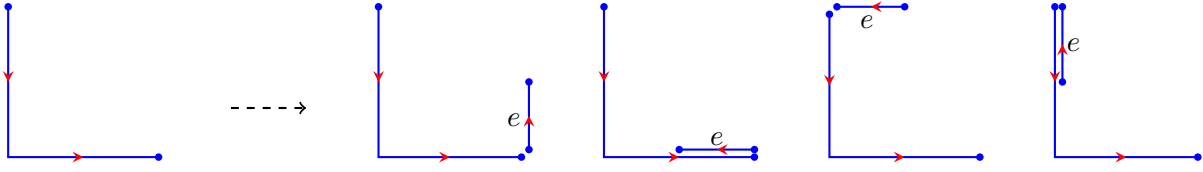
\begin{figure}[h]
\begin{tikzpicture}
\draw[thick,colorline,midarrow]  (0,2) --  (0,0)  -- (2,0);
\node[dot,colorline] at (0,2) {}; \node[dot,colorline] at (2,0) {};
\end{tikzpicture}
\qquad
\begin{tikzpicture}[baseline=-20]
\draw[->,dashed,thick] (0,0) to (1,0);
\end{tikzpicture}
\qquad
\begin{tikzpicture}
\draw[thick,colorline,midarrow]  (0,2) --  (0,0)  -- (1.9, 0);
\node[dot,colorline] at (0,2) {}; \node[dot,colorline] at (1.9 ,0) {};
\draw[thick,colorline,midarrow]  (2, .1) -- (2,1);
\node[dot,colorline] at (2, .1) {}; \node[dot,colorline] at (2,1) {};
\node at (1.8,0.5) {$e$};
\end{tikzpicture}
\qquad
\begin{tikzpicture}
\draw[thick,colorline,midarrow]  (0,2) --  (0,0)  -- (2, 0);
\node[dot,colorline] at (0,2) {}; \node[dot,colorline] at (2 ,0) {};
\draw[thick,colorline,midarrow]  (2, .1) -- (1, .1);
\node[dot,colorline] at (1, .1) {}; \node[dot,colorline] at (2, .1) {};
\node at (1.5,0.25) {$e$};
\end{tikzpicture}
\qquad
\begin{tikzpicture}
\draw[thick,colorline,midarrow]  (0, 1.9) --  (0,0)  -- (2, 0);
\node[dot,colorline] at (0,1.9) {}; \node[dot,colorline] at (2 ,0) {};
\draw[thick,colorline,midarrow] (1, 2) -- (0.1, 2);
\node[dot,colorline] at (1,2) {}; \node[dot,colorline] at (0.1 ,2) {};
\node at (0.5,1.8) {$e$};
\end{tikzpicture}
\qquad
\begin{tikzpicture}
\draw[thick,colorline,midarrow]  (0,2) --  (0,0)  -- (1.9, 0);
\node[dot,colorline] at (0,2) {}; \node[dot,colorline] at (1.9 ,0) {};
\draw[thick,colorline,midarrow] (0.1,1) -- (0.1,2);
\node[dot,colorline] at (0.1,1) {}; \node[dot,colorline] at (0.1,2) {};
\node at (0.25,1.5) {$e$};
\end{tikzpicture}
\caption{Some examples of expansion at an end of a line by $e$. The same pairs of lines with the direction of each $e$ reversed is also considered as this type of expansion.} \label{fig:exp-x}
\end{figure}
For the last type of expansion,
given a line $\ell_1$ with $u(\ell_1)=x$, $v(\ell_1)=y$, and a positive integer $n$, 
we define the
expansion of $\ell_1$ by $n$ null-lines as
\begin{equ}[e:expansion-null]
 (\ddots_x,\cdots,\ddots_x,\ell_1) \in \mL_1\times \mL_1\times \cdots \times \mL_1,
\quad
 (\ell_1, \ddots_y,\cdots,\ddots_y) \in \mL_1\times \mL_1\times \cdots \times \mL_1
\end{equ}
where in each case $\ddots$ repeats $n$ times.
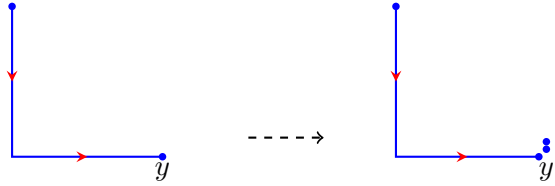
\begin{figure}[h]
\begin{tikzpicture}
\draw[thick,colorline,midarrow]  (0,2) --  (0,0)  -- (2,0);
\node[dot,colorline] at (0,2) {}; \node[dot,colorline] at (2,0) {};
\node at (2,-0.2) {$y$};
\end{tikzpicture}
\qquad
\begin{tikzpicture}[baseline=-20]
\draw[->,dashed,thick] (0,0) to (1,0);
\end{tikzpicture}
\qquad
\begin{tikzpicture}
\draw[thick,colorline,midarrow]  (0,2) --  (0,0)  -- (1.9, 0);
\node[dot,colorline] at (0,2) {}; \node[dot,colorline] at (1.9 ,0) {};
\node[dot,colorline] at (2, .1) {}; \node[dot,colorline] at (2,.2) {};
\node at (2,-0.2) {$y$};
\end{tikzpicture}
\caption{Expansion at an end of a line by a null-line.}
\end{figure}

Note that we distinguish positive versus negative expansions in the first and second types, which will show up in the ``fixed edge'' loop equations, whereas  
 there is no positive versus negative expansion in the third and fourth type,
which will show up in the ``fixed vertex'' loop equations.

\subsubsection{Merger}
\label{sec:Merger}

Given two loops $\ell_0$, $\ell_0'$,
as well as a location $x$ in $\ell_0$ and a location $y$ in $l'_0$, 
we define
 the positive and negative mergers of $\ell_0$ and $l'_0$
	at locations $x,y$
$$
\ell_0\oplus_{x,y} l'_0  \qquad \mbox{and} \qquad
\ell_0\ominus_{x,y} l'_0
$$  
as follows. If $\ell_0=aeb$ and $l'_0=ced$ (where $a,b,c,d$ are paths and $e$ is an edge), define
\begin{equ}[e:merger00-1]
\ell_0\oplus_{x,y} \ell_0' = [aedceb] \in \mL_0,   
\qquad 
\ell_0\ominus_{x,y} \ell_0'=[ac^{-1}d^{-1}b] \in \mL_0.
\end{equ} 
If $\ell_0=aeb$ and $\ell_0'=ce^{-1}d$,  
define \footnote{Note that in this case $\ell_0\oplus_{x,y} \ell_0'$ and $\ell_0'\oplus_{y,x} \ell_0$ are slightly different, namely their orientations are opposite. For $SO(N)$ the Wilson loops with opposite orientations are equal; but for $U(N)$ they are not equal and thus the operation $\oplus$ is not quite symmetric.}
\begin{equ}[e:merger00-2]
\ell_0\oplus_{x,y} \ell_0' =[aec^{-1}d^{-1}eb] \in \mL_0,   
\qquad
\ell_0\ominus_{x,y} \ell_0'=[adcb] \in \mL_0.
\end{equ}
Here (and also in the next paragraph) $x$ and $y$ are the unique location in $\ell_0$ and $\ell'_0$, respectively, where $e$ or $e^{-1}$ occurs and $e$ is the edge occurring at location $x$ in $\ell_0$. 

Fig.~\ref{Positive merger 00} and Fig.~\ref{Negative merger 00} illustrate mergers of two loops.

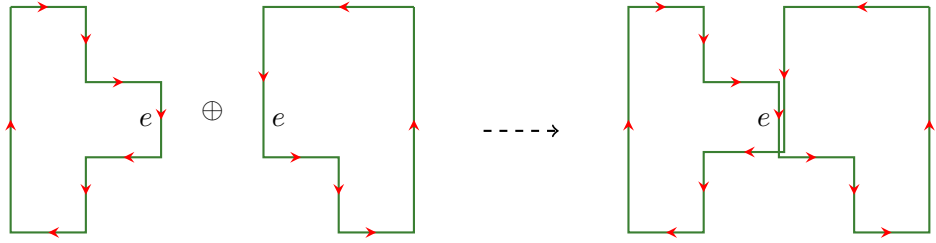
\begin{figure}[h]
\begin{tikzpicture}[baseline=15]
\draw[thick,colorloop,midarrow] (0,2) -- (1,2) -- (1,1) -- (2,1) -- (2,0) -- (1,0) -- (1,-1) -- (0,-1) -- (0,2);
\node at (1.8,0.5) {$e$};
\end{tikzpicture}
$\quad\oplus\quad$
\begin{tikzpicture}[baseline=15]
\draw[thick,colorloop,midarrow] (2,2)  -- (0,2) -- (0,0) -- (1,0) -- (1,-1) -- (2,-1) -- (2,2);
\node at (0.2,0.5) {$e$};
\end{tikzpicture}
\qquad
\begin{tikzpicture}[baseline=5]
\draw[->,dashed,thick] (0,0) to (1,0);
\end{tikzpicture}
\qquad
\begin{tikzpicture}[baseline=15]
\draw[thick,colorloop,midarrow] (4,2)  -- (2.07,2) -- (2.07,0.07) -- (1,0.07) -- (1,-1) -- (0,-1) --(0,2) -- (1,2) -- (1,1) -- (2,1) -- (2,0) -- (3,0) -- (3,-1) -- (4,-1) -- (4,2);
\node at (1.8,0.5) {$e$};
\end{tikzpicture}
\caption{Positive merger of  loops.}
\label{Positive merger 00}
\end{figure}

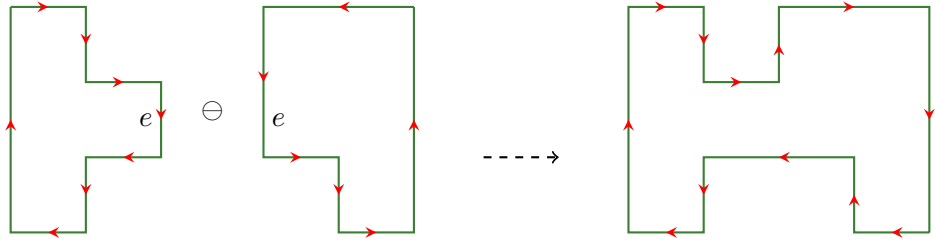
\begin{figure}[h]
\begin{tikzpicture}[baseline=15]
\draw[thick,colorloop,midarrow] (0,2) -- (1,2) -- (1,1) -- (2,1) -- (2,0) -- (1,0) -- (1,-1) -- (0,-1) -- (0,2);
\node at (1.8,0.5) {$e$};
\end{tikzpicture}
$\quad\ominus\quad$
\begin{tikzpicture}[baseline=15]
\draw[thick,colorloop,midarrow]  (2,2)  -- (0,2)  -- (0,0) -- (1,0) -- (1,-1) -- (2,-1) -- (2,2);
\node at (0.2,0.5) {$e$};
\end{tikzpicture}
\qquad
\begin{tikzpicture}[baseline=15]
\draw[->,dashed,thick] (0,0) to (1,0);
\end{tikzpicture}
\qquad
\begin{tikzpicture}[baseline=15]
\draw[thick,colorloop,midarrow] (4,-1) -- (3,-1) -- (3,0)  -- (1,0) -- (1,-1) -- (0,-1) -- (0,2) -- (1,2) -- (1,1) -- (2,1) -- (2,2) -- (4,2) -- (4,-1);
\end{tikzpicture}
\caption{Negative merger of  loops.}
\label{Negative merger 00}
\end{figure}

Given a loop $\ell_0$ and a line $\ell_1$, 
as well as a location $x$ in $\ell_0$ and a location $y$ in $\ell_1$, 
we define
 the positive and negative mergers of $\ell_0$ and $\ell_1$
	at locations $x,y$
$$
\ell_0\oplus_{x,y} \ell_1  \qquad \mbox{and} \qquad
\ell_0\ominus_{x,y} \ell_1
$$  
as follows.
If $\ell_0=aeb$ and $\ell_1=ced$ 
(where $a,b,c,d$ are paths and $e$ is an edge), define 
\footnote{We recall that in the merger of two loops, the letters in $[aedceb] \in \mL_0$ are allowed to be cyclically permuted. 
However, in the merger of a loop and a line here,  the letters in $[cebaed] \in \mL_1$ are not allowed to be cyclically permuted. Moreover, note that $\ell_0\ominus_{x,y} \ell_1$ and $\ell_1\ominus_{y,x} \ell_0$ are almost the same line except that the orientations are opposite.}
\begin{equs}[e:merger01-1]
\ell_0\oplus_{x,y} \ell_1  =\ell_1  \oplus_{y,x} \ell_0 & =[cebaed]   \in \mL_1,  
\\
\ell_0\ominus_{x,y} \ell_1=[d^{-1}bac^{-1}] \in \mL_1,
\qquad&
\ell_1\ominus_{y,x} \ell_0=[c a^{-1}b^{-1}d] \in \mL_1.
\end{equs} 
If $\ell_0=aeb$ and $\ell_1=ce^{-1}d$,  
define 
\begin{equs}[e:merger01-2]
\ell_0\oplus_{x,y} \ell_1  =[d^{-1}eb aec^{-1}]   \in \mL_1,  
\qquad&
\ell_1  \oplus_{y,x} \ell_0 = [c e^{-1} a^{-1} b^{-1} e^{-1} d]   \in \mL_1,  
\\
\ell_0\ominus_{x,y} \ell_1=\ell_1\ominus_{y,x} \ell_0 & =[cbad] \in \mL_1.
\end{equs}
The following pictures illustrate mergers of a loop and a line.

\begin{figure}[h]
\begin{tikzpicture}[baseline=15]
\draw[thick,colorloop,midarrow] (0,2) -- (1,2) -- (1,1) -- (2,1) -- (2,0) -- (1,0) -- (1,-1) -- (0,-1) -- (0,2);
\node at (1.8,0.5) {$e$};
\end{tikzpicture}
$\quad\oplus\quad$
\begin{tikzpicture}[baseline=15]
\draw[thick,colorline,midarrow] (2,2)  -- (0,2) -- (0,0) -- (1,0) -- (1,-1) -- (2,-1);
\node[dot,colorline] at (2,2) {}; \node[dot,colorline] at (2,-1) {};
\node at (0.2,0.5) {$e$};
\end{tikzpicture}
\qquad
\begin{tikzpicture}[baseline=15]
\draw[->,dashed,thick] (0,0) to (1,0);
\end{tikzpicture}
\qquad
\begin{tikzpicture}[baseline=15]
\draw[thick,colorline,midarrow,thick] (4,2)  -- (2.07,2) -- (2.07,0.07) -- (1,0.07) -- (1,-1) -- (0,-1) --(0,2) -- (1,2) -- (1,1) -- (2,1) -- (2,0) -- (3,0) -- (3,-1) -- (4,-1);
\node[dot,colorline] at (4,2) {};  \node[dot,colorline] at (4,-1) {};
\node at (1.8,0.5) {$e$};
\end{tikzpicture}
\caption{Positive merger of a loop and a line.}
\label{Positive merger}
\end{figure}
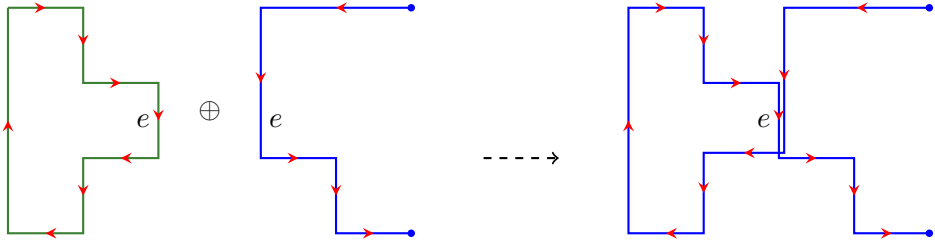

\begin{figure}[h]
\begin{tikzpicture}[baseline=15]
\draw[thick,colorloop,midarrow] (0,2) -- (1,2) -- (1,1) -- (2,1) -- (2,0) -- (1,0) -- (1,-1) -- (0,-1) -- (0,2);
\node at (1.8,0.5) {$e$};
\end{tikzpicture}
$\quad\ominus\quad$
\begin{tikzpicture}[baseline=15]
\draw[thick,colorline,midarrow]  (2,2)  -- (0,2)  -- (0,0) -- (1,0) -- (1,-1) -- (2,-1);
\node[dot,colorline] at (2,2) {}; \node[dot,colorline] at (2,-1) {};
\node at (0.2,0.5) {$e$};
\end{tikzpicture}
\qquad
\begin{tikzpicture}[baseline=15]
\draw[->,dashed,thick] (0,0) to (1,0);
\end{tikzpicture}
\qquad
\begin{tikzpicture}[baseline=15]
\draw[thick,colorline,midarrow] (4,-1) -- (3,-1) -- (3,0)  -- (1,0) -- (1,-1) -- (0,-1) -- (0,2) -- (1,2) -- (1,1) -- (2,1) -- (2,2) -- (4,2);
\node[dot,colorline] at (4,-1) {}; \node[dot,colorline] at (4,2) {}; 
\end{tikzpicture}
\caption{Negative merger of a loop and a line.}
\label{Negative merger}
\end{figure}
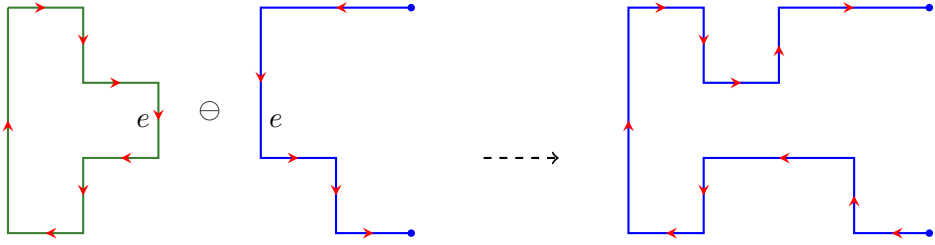

\subsubsection{Switching}
\label{sec:Switching}

Given two lines $\ell_1$, $\ell_1'$,
as well as a location $x$ in $\ell_1$ and a location $y$ in $l'_1$,
we define
 the positive and negative switchings of $\ell_1$ and $\ell_1'$
	at locations $x,y$
$$
\ell_1\oplus_{x,y} \ell'_1  \qquad \mbox{and} \qquad
\ell_1\ominus_{x,y} l'_1
$$  
as follows.
If $\ell_1=aeb$ and $l'_1=ced$ (where $a,b,c,d$ are paths and $e$ is an edge), define
\begin{equs}[e:switching1]
\ell_1\oplus_{x,y} \ell_1'   &= ([aed],[ceb])   \in \mL_1\times \mL_1,
\\
\ell_1\ominus_{x,y} \ell_1' &  = ([ac^{-1}],[d^{-1}b])   \in \mL_1\times \mL_1.
\end{equs} 
If $\ell_1=aeb$ and $\ell_1'=ce^{-1}d$,  
define
\begin{equs}[e:switching2]
\ell_1\oplus_{x,y} \ell_1'   &= ([aec^{-1}],[d^{-1}eb])   \in \mL_1\times \mL_1,
\\
\ell_1\ominus_{x,y} \ell_1'  & = ([ad],[cb])   \in \mL_1\times \mL_1.
\end{equs}
Here $x$ and $y$ are the unique location in $\ell_1$ and $l'_1$, respectively, where $e$ or $e^{-1}$ occurs and $e$ is the edge occurring at location $x$ in $\ell_1$. 

Note that we have used the same notation 
$\oplus_{x,y} $ and $\ominus_{x,y}$
for switchings and mergers, but this will not cause confusion,
since switchings are always for two lines whereas mergers always involve at least a loop. 
The following pictures illustrate switchings of two open lines.

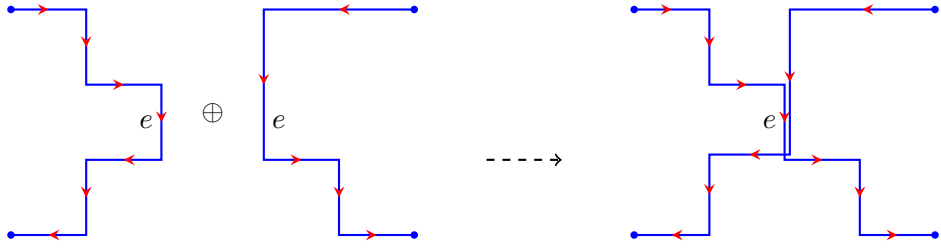
\begin{figure}[h]
\begin{tikzpicture}[baseline=15]
\draw[thick,colorline,midarrow] (0,2) -- (1,2) -- (1,1) -- (2,1) -- (2,0) -- (1,0) -- (1,-1) -- (0,-1);
\node[dot,colorline] at (0,2) {}; \node[dot,colorline] at (0,-1) {};
\node at (1.8,0.5) {$e$};
\end{tikzpicture}
$\quad\oplus\quad$
\begin{tikzpicture}[baseline=15]
\draw[thick,colorline,midarrow] (2,2)  -- (0,2) -- (0,0) -- (1,0) -- (1,-1) -- (2,-1);
\node[dot,colorline] at (2,2) {}; \node[dot,colorline] at (2,-1) {};
\node at (0.2,0.5) {$e$};
\end{tikzpicture}
\qquad
\begin{tikzpicture}[baseline=15]
\draw[->,dashed,thick] (0,0) to (1,0);
\end{tikzpicture}
\qquad
\begin{tikzpicture}[baseline=15]
\draw[thick,colorline,midarrow] (0,2) -- (1,2) -- (1,1) -- (2,1) -- (2,0) -- (3,0) -- (3,-1) -- (4,-1);
\node[dot,colorline] at (0,2) {}; \node[dot,colorline] at (4,-1) {};
\draw[thick,colorline,midarrow] (4,2)  -- (2.07,2) -- (2.07,0.07) -- (1,0.07) -- (1,-1) -- (0,-1);
\node[dot,colorline] at (4,2) {}; \node[dot,colorline] at (0,-1) {};
\node at (1.8,0.5) {$e$};
\end{tikzpicture}
\caption{Positive switching of two open lines.}
\label{Positive switching}
\end{figure}

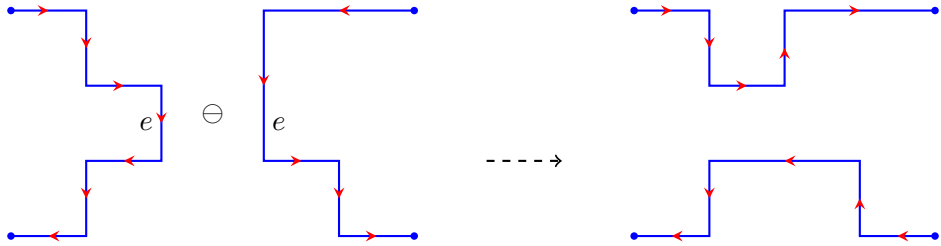
\begin{figure}[h]
\begin{tikzpicture}[baseline=15]
\draw[thick,colorline,midarrow] (0,2) -- (1,2) -- (1,1) -- (2,1) -- (2,0) -- (1,0) -- (1,-1) -- (0,-1);
\node[dot,colorline] at (0,2) {}; \node[dot,colorline] at (0,-1) {};
\node at (1.8,0.5) {$e$};
\end{tikzpicture}
$\quad\ominus\quad$
\begin{tikzpicture}[baseline=15]
\draw[thick,colorline,midarrow]  (2,2)  -- (0,2)  -- (0,0) -- (1,0) -- (1,-1) -- (2,-1);
\node[dot,colorline] at (2,2) {}; \node[dot,colorline] at (2,-1) {};
\node at (0.2,0.5) {$e$};
\end{tikzpicture}
\qquad
\begin{tikzpicture}[baseline=15]
\draw[->,dashed,thick] (0,0) to (1,0);
\end{tikzpicture}
\qquad
\begin{tikzpicture}[baseline=15]
\draw[thick,colorline,midarrow] (0,2) -- (1,2) -- (1,1) -- (2,1) -- (2,2) -- (4,2);
\node[dot,colorline] at (0,2) {}; \node[dot,colorline] at (4,2) {};
\draw[thick,colorline,midarrow]  (4,-1) -- (3,-1) -- (3,0)  -- (1,0) -- (1,-1) -- (0,-1);
\node[dot,colorline] at (4,-1) {}; \node[dot,colorline] at (0,-1) {};
\end{tikzpicture}
\caption{Negative switching of two open lines.}
\label{Negative switching}
\end{figure}

\subsubsection{Deformation}
\label{sec:Deformation}

If a loop or line $\ell'$ is produced by merging a plaquette with a loop or line $\ell$, we will say that $\ell'$ is a  deformation of $\ell$.
If the merger is positive we will say that
$\ell'$ is a positive deformation,
whereas if the merger is negative we will say that $\ell'$ is a negative deformation. 

We will use the notations $\ell\oplus_x p$ and $\ell \ominus_x p$ to  denote the loops or lines obtained by merging $\ell$ and $p$ at locations $x$ and $y$,
where $y$ is the unique location in $p$ where $e$ or $e^{-1}$ occurs, 
and $e$ is the edge occurring at location $x$ in $\ell$.

\begin{figure}[h]
\begin{tikzpicture}[baseline=15]
\draw[thick,colorline,midarrow] (1,-1)-- (0,-1) -- (0,1) -- (3,1) -- (3,-1)  -- (2,-1); 
\node[dot,colorline] at (1,-1) {}; \node[dot,colorline] at (2,-1) {};
\node at (1.5,0.8) {$e$};
\end{tikzpicture}
\qquad
\begin{tikzpicture}[baseline=5]
\draw[->,dashed,thick] (0,0) to (1,0);
\end{tikzpicture}
\qquad
\begin{tikzpicture}[baseline=15]
\draw[thick,colorline,midarrow] (1,-1)-- (0,-1) -- (0,1) --(1,1) -- (1,2) -- (2,2) -- (2,1)-- (3,1) -- (3,-1)  -- (2,-1); 
\node[dot,colorline] at (1,-1) {}; \node[dot,colorline] at (2,-1) {};
\node at (1.5,0.8) {$e$};
\end{tikzpicture}
\qquad
\begin{tikzpicture}[baseline=15]
\draw[thick,colorline,midarrow] (1,-1)-- (0,-1) -- (0,1) -- (2,1) -- (2,0) -- (1,0) -- (1,1.07) -- (2,1.07)-- (3,1.07) -- (3,-1)  -- (2,-1); 
\node[dot,colorline] at (1,-1) {}; \node[dot,colorline] at (2,-1) {};
\node at (1.5,0.8) {$e$};
\end{tikzpicture}
\caption{Negative and positive deformations of a line.}
\end{figure}
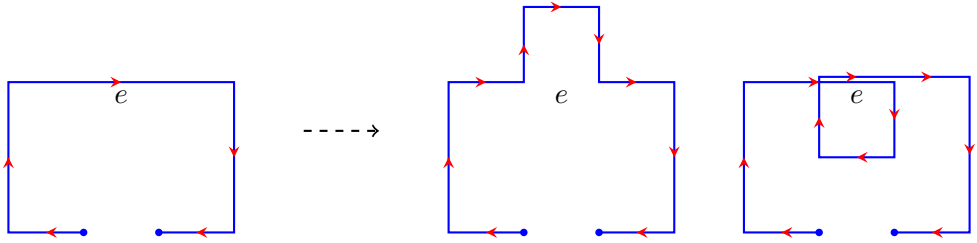

\subsubsection{Breaking}
\label{sec:Breaking}

Given a loop $\ell_0$, a line $\ell_1$,
and a  location $x$ in $\ell_i$ ($i\in \{0,1\}$),
we define the positive  and negative breaking of $\ell_i$ at location $x$
$$
\nparallel_{x}^\pm \ell_i
$$
as follows.
If $\ell_i=ae b$  (where $a, b$ are paths and $e$ is an edge), 
define negative breaking
\begin{equ}[e:breaking1]
\nparallel_{x}^- \ell_0 = [ ba] \in \mL_1,
\qquad
\nparallel_{x}^- \ell_1 = ( [a],[ b])\in \mL_1\times \mL_1,
\end{equ}
and positive breaking
\begin{equ}[e:breaking2]
\nparallel_{x}^+ \ell_0 = [e ba e] \in \mL_1,
\qquad
\nparallel_{x}^+ \ell_1 = ( [a e],[e b])\in \mL_1\times \mL_1.
\end{equ}
Here $x$ is the unique location in $\ell_i$ where $e$  occurs.

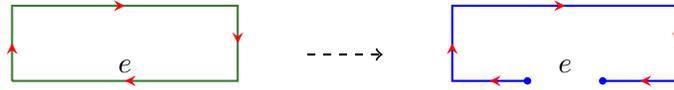
\begin{figure}[h]
\begin{tikzpicture}[baseline=15]
\draw[thick,colorloop,midarrow] (0,0) -- (0,1) -- (3,1) -- (3,0) -- (0,0);
\node at (1.5,0.2) {$e$};
\end{tikzpicture}
\qquad
\begin{tikzpicture}[baseline=5]
\draw[->,dashed,thick] (0,0) to (1,0);
\end{tikzpicture}
\qquad
\begin{tikzpicture}[baseline=15]
\draw[thick,colorline,midarrow] (1,0)-- (0,0) -- (0,1) -- (3,1) -- (3,0)  -- (2,0); 
\node[dot,colorline] at (1,0) {}; \node[dot,colorline] at (2,0) {};
\node at (1.5,0.2) {$e$};
\end{tikzpicture}
\caption{Negative breaking of a loop, resulting in a line.}
\end{figure}

\begin{figure}[h]
\begin{tikzpicture}[baseline=15]
\draw[thick,colorline,midarrow]  (3,1) -- (3,0)  -- (0,0) -- (0,1);
\node[dot,colorline] at (3,1) {}; \node[dot,colorline] at (0,1) {};
\node at (1.5,0.2) {$e$};
\end{tikzpicture}
\qquad
\begin{tikzpicture}[baseline=5]
\draw[->,dashed,thick] (0,0) to (1,0);
\end{tikzpicture}
\qquad
\begin{tikzpicture}[baseline=15]
\draw[thick,colorline,midarrow]  (3,1) -- (3, .07)  -- (1, .07);
\node[dot,colorline] at (3,1) {}; \node[dot,colorline] at (1, .07) {};
\draw[thick,colorline,midarrow]  (2,0) -- (0,0) -- (0,1);
\node[dot,colorline] at (2,0) {}; \node[dot,colorline] at (0,1) {};
\node at (1.5,0.2) {$e$};
\end{tikzpicture}
\caption{Positive breaking of an open line, resulting in two lines.}
\end{figure}
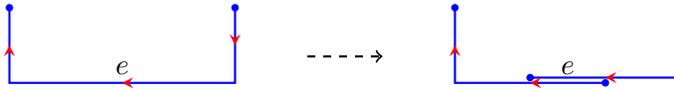

\subsubsection{Twisting}
\label{sec:Twisting}

Let $i\in \{0,1\}$. Given a loop or line $\ell\in \mL_i$, 
we write 
$\propto_{x,y} \ell $ for the negative twisting
	if $l$ contains an edge $e$ at both $x$ and $y$,
	or positive twisting if $\ell$ contains an edge $e$ at location $x$ and $e^{-1}$ at location $y$, which are defined as follows. 
	For $\ell=aebec$, define 
	the negative twisting 
	\begin{equ}[e:twisting1]
	\propto_{x,y} \ell \eqdef [ab^{-1}c] \in \mL_i.
	\end{equ}
	 For  $\ell=aebe^{-1}c$,  define 
	the positive twisting 
\begin{equ}[e:twisting2]
\propto_{x,y} \ell \eqdef [aeb^{-1}e^{-1}c] \in \mL_i.
\end{equ}

\begin{figure}[h]
\begin{tikzpicture}[baseline=15]
\draw[thick,colorline,midarrow] (1,-1)-- (0,-1) -- (0,1) -- (2,1) -- (2,0) -- (1,0) -- (1,1.07) -- (2,1.07)-- (3,1.07) -- (3,-1)  -- (2,-1); 
\node[dot,colorline] at (1,-1) {}; \node[dot,colorline] at (2,-1) {};
\node at (1.5,0.8) {$e$};
\end{tikzpicture}
\qquad
\begin{tikzpicture}[baseline=10]
\draw[->,dashed,thick] (0,0) to (1,0);
\end{tikzpicture}
\qquad
\begin{tikzpicture}[baseline=15]
\draw[thick,colorline,midarrow] (1,-1)-- (0,-1) -- (0,1) --(1,1) -- (1,0) -- (2,0) -- (2,1)-- (3,1) -- (3,-1)  -- (2,-1); 
\node[dot,colorline] at (1,-1) {}; \node[dot,colorline] at (2,-1) {};
\node at (1.5,0.8) {$e$};
\end{tikzpicture}
\caption{Negative twisting of a line.}
\end{figure}
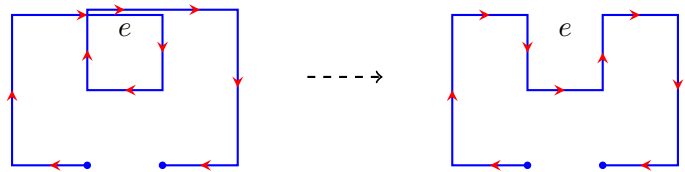

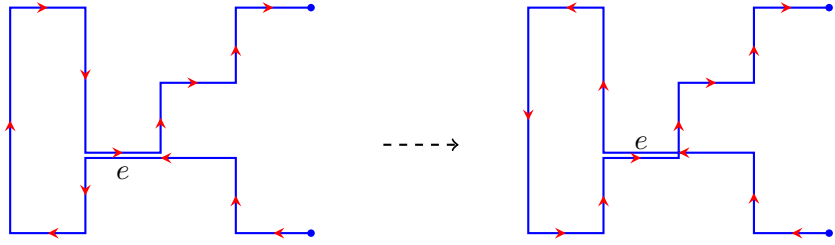
\begin{figure}[h]
\begin{tikzpicture}[baseline=15]
\draw[thick,colorline,midarrow] (4,-1) -- (3,-1) -- (3,0)  -- (1,0) -- (1,-1) -- (0,-1) -- (0,2) -- (1,2) -- (1,0.07) -- (2,0.07) -- (2,1) -- (3,1) -- (3,2) -- (4,2);
\node[dot,colorline] at (4,-1) {}; \node[dot,colorline] at (4,2) {}; 
\node at (1.5,-0.2) {$e$};
\end{tikzpicture}
\qquad
\begin{tikzpicture}[baseline=10]
\draw[->,dashed,thick] (0,0) to (1,0);
\end{tikzpicture}
\qquad
\begin{tikzpicture}[baseline=15]
\draw[thick,colorline,midarrow] (4,-1) -- (3,-1) -- (3,0.07)  -- (1,0.07) -- (1,2) -- (0,2) -- (0,-1) -- (1,-1) -- (1,0) -- (2,0) -- (2,1) -- (3,1) -- (3,2) -- (4,2);
\node[dot,colorline] at (4,-1) {}; \node[dot,colorline] at (4,2) {}; 
\node at (1.5,0.2) {$e$};
\end{tikzpicture}
\caption{Positive twisting of a line.}
\end{figure}

\subsubsection{Gluing}
\label{sec:Gluing}
Given $\ell_1\in \mL_1$ with $u(\ell_1)=v(\ell_1)$, we define
$$
\bigtriangleup \ell_1 = \ell_1\in \mL_0.
$$
Given $\ell_1,\ell_2\in \mL_1$ with $v(\ell_1) = u(\ell_2)$, we define 
$$
\ell_1 \bigtriangleup \ell_2 = [\ell_1 \ell_2] \in \mL_1.
$$
There is an additional type of gluing,
called $\mR$-gluing, 
which is  relevant in the real case but not in the complex case, defined as following.
Given $\ell_1,\ell_2\in \mL_1$ with $u(\ell_1) = u(\ell_2)$, we define 
$$
\ell_1 \bigtriangleup \ell_2 = [\ell_1^{-1} \ell_2] \in \mL_1.
$$
If $v(\ell_1) = v(\ell_2)$, we define 
$$
\ell_1 \bigtriangleup \ell_2 = [\ell_1 \ell_2^{-1}] \in \mL_1.
$$

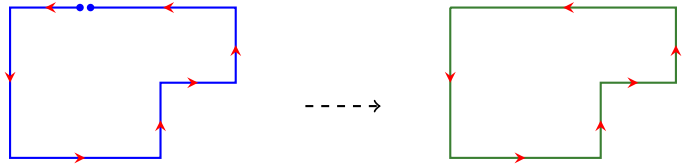
\begin{figure}[h]
\begin{tikzpicture}
\draw[thick,colorline,midarrow] (-0.07,0) -- (-1,0)  -- (-1,-2)  -- (1,-2) -- (1,-1) -- (2,-1) -- (2,0) -- (0.07,0);
\node[dot,colorline] at (-0.07,0) {};
\node[dot,colorline] at (0.07,0) {};
\end{tikzpicture}
\qquad
\begin{tikzpicture}[baseline=-20]
\draw[->,dashed,thick] (0,0) to (1,0);
\end{tikzpicture}
\qquad
\begin{tikzpicture}
\draw[thick,colorloop,midarrow]  (-1,0)  -- (-1,-2)  -- (1,-2) -- (1,-1) -- (2,-1) -- (2,0) -- (-1,0);
\end{tikzpicture}
\caption{Gluing an open line into a loop.}
\label{fig:Gluing-line-into-loop}
\end{figure}

\begin{figure}[h]
\begin{tikzpicture}[baseline=-10]
\draw[thick,colorline,midarrow] (0,0) -- (1,0) -- (1,1) -- (2,1)  -- (2,-1);
\node[dot,colorline] at (0,0) {};
\node[dot,colorline] at (2,-1) {};
\node at (1.8,-0.5) {$e$};
\end{tikzpicture}
$\quad\bigtriangleup\quad$
\begin{tikzpicture}[baseline=5]
\draw[thick,colorline,midarrow] (0,0) -- (0,1) -- (1,1) -- (1,0) -- (2,0);
\node[dot,colorline] at (0,0) {};
\node[dot,colorline] at (2,0) {};
\node at (0.2,0.5) {$e$};
\end{tikzpicture}
\qquad
\begin{tikzpicture}[baseline=-15]
\draw[->,dashed,thick] (0,0) to (1,0);
\end{tikzpicture}
\qquad
\begin{tikzpicture}[baseline=-15]
\draw[thick,colorline,midarrow]  (0,0) -- (1,0) -- (1,1) -- (2,1) -- (2,0) -- (3,0) -- (3,-1) -- (4,-1);
\node[dot,colorline] at (0,0) {};
\node[dot,colorline] at (4,-1) {};
\end{tikzpicture}
\caption{Gluing two lines into one line.}
\end{figure}
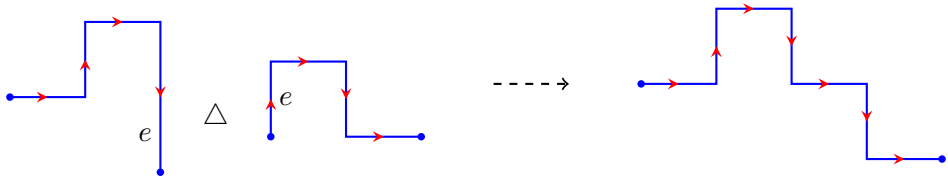

\subsubsection{Extension}
\label{sec:Extension}

Just like deformation is defined by merging a string with a plaquette,
extension of a line $\ell$ is defined by gluing an edge with $\ell$.
More precisely:
 
Given a line $\ell_1$ and an edge $e$ where $v(\ell_1)=u(e)$, we define the  extension by
\begin{equ}[e:extension1]
 \ell_1\bigtriangleup e = [\ell_1 e] \in \mL_1.
\end{equ}
Given a line $\ell_1$ and an edge $e$ where $u(\ell_1)=v(e)$, we define the  extension by
\begin{equ}[e:extension2]
e \bigtriangleup \ell_1 = [e \ell_1] \in \mL_1.
\end{equ}

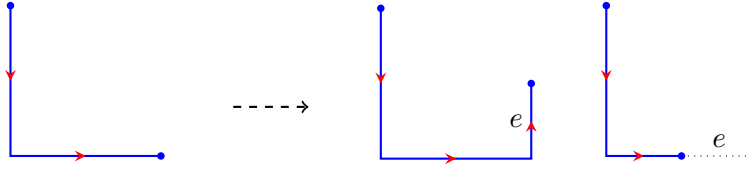
\begin{figure}[h]
\begin{tikzpicture}
\draw[thick,colorline,midarrow] (0,2) -- (0,0)  -- (2,0);
\node[dot,colorline] at (0,2) {}; \node[dot,colorline] at (2,0) {};
\end{tikzpicture}
\qquad
\begin{tikzpicture}[baseline=-20]
\draw[->,dashed,thick] (0,0) to (1,0);
\end{tikzpicture}
\qquad
\begin{tikzpicture}
\draw[thick,colorline,midarrow]  (0,2) -- (0,0)  -- (2,0) -- (2,1);
\node[dot,colorline] at (0,2) {}; \node[dot,colorline] at (2,1) {};
\node at (1.8,0.5) {$e$};
\end{tikzpicture}
\qquad
\begin{tikzpicture}
\draw[thick,colorline,midarrow]  (0,2) -- (0,0) -- (1,0);
\node[dot,colorline] at (0,2) {}; \node[dot,colorline] at (1,0) {};
\draw[dotted] (1,0) -- (2,0);
\node at (1.5,0.2) {$e$};
\end{tikzpicture}
\caption{Some examples of extension at the end point $v(l_1)$.}
\end{figure}

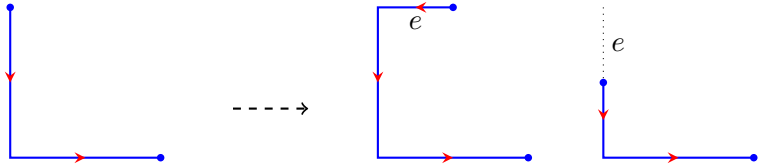
\begin{figure}[h]
\begin{tikzpicture}
\draw[thick,colorline,midarrow] (0,2) -- (0,0)  -- (2,0);
\node[dot,colorline] at (0,2) {}; \node[dot,colorline] at (2,0) {};
\end{tikzpicture}
\qquad
\begin{tikzpicture}[baseline=-20]
\draw[->,dashed,thick] (0,0) to (1,0);
\end{tikzpicture}
\qquad
\begin{tikzpicture}
\draw[thick,colorline,midarrow] (1,2) -- (0,2) -- (0,0) -- (2,0);
\node[dot,colorline] at (1,2) {}; \node[dot,colorline] at (2,0) {};
\node at (0.5,1.8) {$e$};
\end{tikzpicture}
\qquad
\begin{tikzpicture}
\draw[thick,colorline,midarrow]  (0,1)-- (0,0)  -- (2,0);
\node[dot,colorline] at (0,1) {}; \node[dot,colorline] at (2,0) {};
\draw[dotted] (0,2) -- (0,1);
\node at (0.2,1.5) {$e$};
\end{tikzpicture}
\caption{Some examples of extension at the beginning point $u(\ell_1)$.}
\end{figure}

\subsubsection{Sets of operations}

Now we extend the definitions of the operations to string sequences. 
Let $s$ and $s'$ be two string sequences.

We say that the loop sequence $s'$ is obtained from splitting the loop sequence $s$ if exactly two components of $s'$ arise from splitting a single string in $s$. 

We say that the loop sequence $s'$
is a deformation / twisting / extension of $s$
 if $s$ is obtained by deforming / twisting / extending one of the  component  of $s$.

 We say that $s'$ is a merger of $s$ if $s'$ is obtained by merging two strings of $s$.
We say that $s'$ is a switching of $s$
if $s'$ is obtained by switching two lines of $s$.
 
 We say that $s'$ is a breaking of $s$,
  if exactly two line components of $s'$ arise from breaking a  line in $s$,
 or exactly one line component of $s'$ arises from breaking a  loop in $s$.
 
 We say that  $s'$ is obtained from gluing  $s$,
 if exactly one loop component of $s'$ arises from gluing a single line in $s$,
or exactly one line component of $s'$ arises from gluing two lines in $s$.
 
We say that $s'$ is an expansion of $s$,
if exactly two  components of $s'$ arise from  expanding a component in $s$ by a plaquette or an edge,
or exactly $n+1$  components of $s'$ arise from  expanding a component in $s$ 
by $n$ null-edges.
For $e\in E$, $x\in \Lambda$, define 
\begin{align*}
 \mD_{e}(s) & = \{s' \;:\; \mbox{$s'$ is a deformation of $s$ at $e$ or $e^{-1}$}\}
 \\
 \mB_{e}(s) & = \{s' \;:\; \mbox{$s'$ is a breaking of $s$ at $e$ or $e^{-1}$}\}
\\
 \mS_{e}(s) & =  \{s' \;:\; \mbox{$s'$ is a  splitting of $s$ at $e$ or $e^{-1}$}\}
\\
 \mT_{e}(s) & =  \{s' \;:\; \mbox{$s'$ is a twisting of $s$ at $e$ or $e^{-1}$}\}
\\
 \mX_{e}(s) & =  \{s' \;:\; \mbox{$s'$ is a switching of $s$ at $e$ or $e^{-1}$}\}
\\
 \mM_{e}(s)& = \{s' \;:\; \mbox{$s'$ is a merger of $s$ at $e$ or $e^{-1}$}\}
\\
 \mG_x(s) & =  \{s' \;:\; \mbox{$s'$ is a gluing (but not $\mR$-gluing) of $s$ at $x$}\}
 \\
 \mG^{\mR}_x(s) & =  \{s' \;:\; \mbox{$s'$ is an $\mR$-gluing of $s$ at $x$}\}
  \\
 \mE_{\text{ext},x}(s) & =  \{s' \;:\; \mbox{$s'$ is an extension of $s$ at $x$}\}
 \\
\mE_{e,\square} (s) & = \{s' \;:\; \mbox{$s'$ is a (first type)  expansion of $s$ at $e$ by a plaquette}\}
\\
\mE_{e,\parallel} (s) & = \{s' \;:\; \mbox{$s'$ is a (second type) expansion of $s$ at $e$ by an edge}\}
\\
\mE_{x}(s) & = \{s' \;:\; \mbox{$s'$ is an expansion of $s$ at $x$ by an edge}\}
\\
\mE^k_{x}(s) & =  \{s' \;:\; \mbox{$s'$ is an expansion of $s$ at $x$ by $k$ null-lines}\}
\end{align*}   
 For each of the above sets $\mO_{e}(s)$ where $\mO  \in \{\mD,\mB,\mS,\mT,\mX,\mM,\mE_{e,\parallel},\mE_{e,\square}\}$, we can decompose them into a disjoint union based on whether the operation is positive or negative, and we denote this as $\mO_{e}(s)=:\mO_{e}^{-}(s)\cup \mO_{e}^{+}(s)$.  Furthermore, we define two sets $\mM^{+}_{e,U}(s) \subset \mM^{+}_e(s)$ and $\mM^{-}_{e,U}(s) \subset \mM^{-}_e(s)$:
  
(1) The set $\mM^{+}_{e,U}(s)$ consists of positive mergers resulting from an edge $e$ appearing in both of the two merged strings (namely, from the first identity of \eqref{e:merger00-1} or the first identity of \eqref{e:merger01-1}).

(2) The set $\mM^{-}_{e,U}(s)$ consists of negative mergers where an edge $e$ occurs in one string and $e^{-1}$ in the other (namely, from the second identity of \eqref{e:merger00-2} or the last identity of \eqref{e:merger01-2}). 

We set $\mM^{+}_{e,\backslash U}(s) \eqdef \mM^{+}_e(s)\backslash \mM^{+}_{e,U}(s)$ and $\mM^{-}_{e,\backslash U}(s) \eqdef \mM^{-}_e(s)\backslash \mM^{+}_{e,U}(s)$. 

Also we define two sets $\mX^{+}_{e,U}(s) \subset \mX^{+}_e(s)$ 
and $\mX^{-}_{e,U}(s) \subset \mX^{-}_e(s)$:

(1) The set $\mX^{+}_{e,U}(s)$ consists of positive switchings resulting from an edge $e$ appearing in both of the two switched lines  (namely, from the first identity of \eqref{e:switching1}).

(2) The set $\mX^{-}_{e,U}(s)$ consists of negative switchings where an edge $e$ occurs in one line and $e^{-1}$ in the other (namely, from the second identity of \eqref{e:switching2}). 

We set $\mX^{+}_{e,\backslash U}(s) \eqdef \mX^{+}_e(s)\backslash \mX^{+}_{e,U}(s)$ and $\mX^{-}_{e,\backslash U}(s) \eqdef \mX^{-}_e(s)\backslash \mX^{+}_{e,U}(s)$. 
 Remark that for deformation, breaking, splitting, twisting, switching, merger,
  in choosing the words {\it positive} versus {\it negative} (which seems 
 arbitrary a priori), 
   there is an interesting coincidence:
1) For all the {\it positive} operations, $e$ or $e^{-1}$ appears twice, whereas 
 for all the {\it negative} operations, neither  $e$ nor $e^{-1}$ appears.
See \eqref{e:splitting1}, \eqref{e:splitting2}, 
\eqref{e:merger00-1}, \eqref{e:merger00-2}, \eqref{e:merger01-1}, \eqref{e:merger01-2}, 
 \eqref{e:switching1},  \eqref{e:switching2},
  \eqref{e:breaking1},  \eqref{e:breaking2},
   \eqref{e:twisting1},     \eqref{e:twisting2}.
2) In the Makeenko--Migdal equations,
all the  {\it positive} operations on the RHS come with negative coefficients,
and 
all the  {\it negative} operations on the RHS come with positive coefficients.

 We use the following shorthand notation (which we borrow from \cite{Cao2025area}): for each operation $\mO$ above with both a positive and negative version and any constant $C$, we write
\begin{equation}
\mp C\sum_{s' \in \mO^\pm_{e}(s)}(\cdots) 
\eqdef -C\sum_{s' \in \mO_{e}^{+}(s)}(\cdots)+C\sum_{s' \in \mO_{e}^{-}(s)}(\cdots).
\end{equation}
\subsubsection{Remarks on the sets}  \label{sec:Remarks} 

In the above definitions, $e\in E$ and its inverse may appear in $s$
multiple times, i.e. in multiple locations of $s$.
The set $\mD^-_e(s)$ then contains the negative deformations 
at all these locations, and the same for the other sets. For example,
if $e$ or $e^{-1}$ appear at locations of $x_1,\cdots,x_r$ in a string collection $s$,
one can write a disjoint union
\begin{equ}[e:D-union]
 \mD^{-}_e(s) 
=
\cup_{i=1}^r \{\mbox{negative deformations of $s$ at location $x_i$}\}
\end{equ}
Similarly,
\begin{equ}[e:S+union]
 \mS^{+}_e(\ell) 
=
\cup_{i\neq j} \{\mbox{positive splittings  of $\ell$
at locations $x_i,x_j$}\}
\end{equ}
where, by definition of positive splitting \eqref{e:splitting1}, 
for fixed $i,j$ with $i\neq j$, 
the set 
$\{\cdots\}$ on the RHS 
 is non-empty if 
either we have $e$ at both locations $x_i,x_j$ or $e^{-1}$ at both locations $x_i,x_j$.
Also,
the union of  $\mD^{-}_{e}(s)$, for example,
over all $e$  in $s$,
is the set $\mD^{-}(s)$ in \cite{Cha,SSZloop}. 
   
Moreover,
in the above definitions, the lattice site $x$ may also appear in $s$
multiple times, namely,
a line in $s$ may have both its beginning and  ending points equal to $x$,
or, multiple lines in $s$ may share $x$ as  their beginning or ending point,
or both situations may occur simultaneously. 
The set $ \mE_{\text{ext},x}(s)$  then consists of all  extensions taken 
at  the beginning or ending points of these lines.
The same is understood  for $\mE_{\text{exp},x}(s)$ and $\mE_{\text{exp},x}^k(s)$.

For example,
if a line $\ell \in \mL_{1}$ has 
both its beginning point and ending point equal to $x\in \Lambda$ (see  Figure~\ref{fig:Gluing-line-into-loop}),
writing the
beginning point $y$
and the ending point $z$ (namely $x=y=z$), we have a disjoint union
\begin{equation}\label{e:Exp-union}
\mE_{\text{exp},x}(\ell)
= \{\mbox{expansions of $\ell$ at $y$}\}
\cup  \{\mbox{expansions of $\ell$ at $z$}\}\;.
\end{equation}

\br\label{rem:double-counting}
As  in \cite[Sec~2.2]{Cha} (below definitions of $\mS^-,\mS^+$ therein),
and also \cite{SSZloop},
	if a loop $\ell$ has a splitting at $x$ and $y$, then it also has a splitting at $y$ and $x$, and they should be counted as two distinct splittings of $\ell$. 
The same applies to  twisting and we will keep this in mind below.
\er

\section{Proof via  It\^o's formula}
\label{sec:Ito}
For a string $\ell$, we define $\phi(\ell):= \E W_{\ell}$ and for a collection of strings $s=(\ell_{1},\cdots, \ell_{k})$ we define
\begin{equation}
W_s \eqdef \prod_{i=1}^{k}W_{\ell_{i}},
\qquad
\phi(s) \eqdef \E W_s.
\end{equation}
Here $W_\ell$ is the Wilson loop or line observable defined as in Definitions~\ref{def:Wloops} and \ref{def:Wlines};
for a null-line at $x$ it is understood as $\Phi_x^*\Phi_x$.
Note that as we already mentioned in Remark~\ref{rem:2},
the above definition of $\phi(s)$ 
is different from \cite{Cha,SSZloop} since we do not normalize  it by $1/N^k$.
This will make our notation lighter, and it will be easy
to find the loop equations for the normalized $\phi$ by post-processing on our un-normalized version.

The general idea to prove 
Makeenko--Migdal equations for $\phi(s)$ is as follows. 
We differentiate $W_s$ and apply It\^o formula 
to the conditional SDE \eqref{SDE} for $Q_e$ with a fixed edge $e$
or the conditional SDE \eqref{SDE1} for $\Phi_x$ a fixed site $x$.
These SDEs have drift terms and martingale terms.
So $\dif W_s$ will be expressed as terms involving drifts, martingales, and It\^o corrections. The  It\^o corrections arise from co-variations of martingales in the SDEs,
since $e$ or $x$ may appear multiple times within a string or between different strings. 
Assuming the SDE is in stationarity,
taking expectation kills the $\dif W_s$ and martingale terms, resulting in the desired loop equations.

We make a slight departure in exposition compared to our prior work \cite{SSZloop}.  Rather than consider the most general case of collections of strings from the start, we proceed from the specific to the more general.  In fact, in Section \ref{sec:simple} we start with simple loops (and lines), which are already very interesting  on their own, carrying important information  about the YMH measure on both small and large scales.  Unfortunately, due to the operations that arise, we cannot obtain a closed system within the simple Wilson strings, forcing us to consider geometrically more complicated observables with various self-crossings.  These are considered in Section \ref{sec:nonSimple} and lead to additional operations, but still not a closed system.  Finally, by considering collections of strings in Section \ref{sec:collections}, we obtain a closed system and establish the most general version of our result.

\subsection{Simple strings} \label{sec:simple}
 
A string is called simple if it has no repeated edges \footnote{We also exclude that both $e$ and $e^{-1}$ appear in the string.}, and the beginning and ending points are distinct in the case that it is a line.
We first demonstrate our method with simple strings. 
For simple strings, all the operations that arise are a result of the gradients of $\mathcal{S}_{\YMH}$ (recall \eqref{e:SSV}), which are given in \eqref{e:grad_e}, \eqref{e:grad_x}.  

\begin{lemma}\label{lem:simple-e}
	Let $G\in \{SO(N),U(N), SU(N)\}$. 
	If $\ell$ is a simple string that contains the edge $e\in E_\Lambda$,  then 
	\begin{align}
		-2c_\mfg \phi(\ell)
		=\mp\beta \sum_{\ell' \in \mD^\pm_e(\ell) } \phi(\ell')
		\mp\kappa \sum_{s \in \mB^\pm_e(\ell) } \phi(s)\mp\beta \nu \sum_{s \in \mE^\pm_{e,\square}(\ell) } \phi(s)\mp\kappa\nu \sum_{s \in \mE^\pm_{e,\parallel}(\ell) } \phi(s)
		\label{e6}.
	\end{align}
Here $c_\mfg$ is as in \eqref{e:c_g} and $\nu$ is as in \eqref{e:lambda-mu-nu}. \end{lemma}
Note that the sets $\mB^{+}_e(\ell)$ and $\mB^{-}_e(\ell)$
are now  singletons in \eqref{e6}. 

\begin{proof}
We consider the stationary dynamic obtained from conditioning on all values of $\Phi$ and all values of $Q$ except for $Q_{e}$.  
We will obtain \eqref{e6} as an immediate consequence of the following: for any simple string $\ell$ and $G=SO(N), U(N)$
	\begin{equation}
		\dif W_{\ell} 
		-c_\mfg W_{\ell}\dif t
		=\Big[\mp\frac\beta2 \sum_{\ell' \in \mD_{e}(\ell) }W_{\ell'}\mp\frac\kappa2\sum_{s \in \mB_e(\ell)}W_{s}\Big]\dif t+\dif M_{\ell} \label{e7}, 
	\end{equation}
	where $M_{\ell}$ is a martingale.  To establish \eqref{e7}, we note that due to the conditioning, the differential falls only on the variable $Q_{e}$ in $W_\ell$, which appears exactly once since the string is simple.  After inserting the dynamic in Lemma~\ref{lem:DM}, it suffices to calculate the result obtained from replacing the sole instance of $Q_{e}$ by $\nabla_{e}\mathcal{S}_{\text{YMH}}(Q,\Phi)$ within $W_{\ell}$.  
	
	If $\ell$ is a loop, then without loss of generality we may present it as $\ell=ae$ for some path $a$ and correspondingly $W_{\ell}=\text{Tr}(Q_{a}Q_{e})$.  As in \cite{SSZloop}, \footnote{with proper adjustments according to Remark~\ref{rem:2}} it holds
	\begin{equation}
		\text{Tr}\big ( Q_{a}\nabla_{e}\mathcal{S}_{1} \big )
		=\mp\frac{1}{2}\beta \sum_{\ell' \in \mD_e(\ell)}W_{\ell'}
		\label{e4}.
	\end{equation}
	In addition, by \eqref{e:grad_e} and recalling that 
	$eae\in \mL_1$ is the positive breaking of $\ell\in \mL_0$ defined in \eqref{e:breaking2}
	and $a\in \mL_1$ is the negative breaking of $\ell\in \mL_0$ defined in \eqref{e:breaking1},
	it holds
	\begin{align}
		\text{Tr}\big ( Q_{a}\nabla_e\mathcal{S}_{2} \big )
		=-\frac\kappa2    \text{Tr}\big(Q_{a}Q_{e} \Phi_{w}\Phi_{z}^{*}Q_{e} -Q_{a}\Phi_{z} \Phi_{w}^{*}  \big) 
		=-\frac\kappa2  W_{eae }+\frac\kappa2 W_{a}=
		\mp\frac\kappa2 \sum_{s \in \mB_e(\ell)}W_{s}, \nonumber
	\end{align}
	for $e=(z,w)$. 
	
	Similarly, if $\ell$ is a line containing $e$ with endpoints $x \neq y$
	we may present it as $\ell=aeb$ and $W_{\ell}=\Phi_{x}^{*}Q_{a}Q_{e}Q_{b}\Phi_{y}$.  
	The contribution from $\nabla_{e}\mathcal{S}_{1}$ also gives \eqref{e4}.
	Furthermore, recalling that 
	$(ae,eb)\in \mL_1\times \mL_1$ is the positive breaking of $\ell\in \mL_1$ defined in \eqref{e:breaking2}
	and $(a,b)\in \mL_1\times \mL_1$ is the positive breaking of $\ell\in \mL_1$ defined in \eqref{e:breaking1},
	we see that
	\begin{align}
		\text{Tr}\big (\Phi_{x}^{*} Q_{a} \, \nabla_e\mathcal{S}_{2}\,Q_{b}\Phi_{y} \big )
		&=-\frac\kappa2    \text{Tr}\big(\Phi_{x}^{*} Q_{a}Q_{e} \Phi_{w}\Phi_{z}^{*}Q_{e}Q_{b}\Phi_{y} -\Phi_{x}^{*}Q_{a} \Phi_{z}\Phi_{w}^{*}Q_{b}\Phi_{y}  \big) \nonumber \\
		&=-\frac\kappa2 W_{ae}W_{eb}+\frac\kappa2 W_{a}W_{b}=
		\mp\frac\kappa2 \sum_{s \in \mB_e(\ell)}W_{s}\label{e5},
	\end{align}
	for $e=(z,w)$. 
	Based on the above equalities, we readily obtain \eqref{e7} by It\^{o}'s formula applied to the conditional dynamic.
	Taking expectation w.r.t. the dynamics for the conditional probability measure $\mu_{Q_{\backslash e},\Phi}$ and then  w.r.t. $\nu_e$ from \eqref{mea:dis}, we obtain \eqref{e6} for $G=SO(N)$ and $U(N)$. 
	
	For $G=SU(N)$, where $\nu=\frac1N$, the result follows the same line. The only modification lies in the trace component within the terms $\nabla_e \cS_1$ and $\nabla_e\cS_2$ in \eqref{e:grad_e}. 
As in \cite{SSZloop} the contribution from $\nabla_e\cS_1$ gives the first   type of expansion (by a plaquette)
 described in Section~\ref{sec:Expansion}. 
 Similarly by  \eqref{e:grad_e} the contribution from $\nabla_e\cS_2$,
say in the case of Wilson line  $W_{\ell}=\Phi_{x}^{*}Q_{a}Q_{e}Q_{b}\Phi_{y}$, gives
\begin{align*}
	&\Phi_{x}^{*}Q_{a} \Big(
	\frac{-\kappa}{2N} \tr (\Phi_z\Phi_w^* Q_e^* - Q_e \Phi_w \Phi_z^*) I_N Q_e
	\Big) Q_{b}\Phi_{y}
	\\
	&=
	\frac{-\kappa}{2N} \tr (\Phi_z\Phi_w^* Q_e^* - Q_e \Phi_w \Phi_z^*) 
	\Big(\Phi_{x}^{*}Q_{a}  Q_e
	Q_{b}\Phi_{y}\Big) 
	=
	\mp\frac\kappa{2N} \sum_{s \in \mE_{e,\parallel}(\ell)}W_{s}
	\end{align*}
for $e=(z,w)$, leading to expansions by edge-lines.
\end{proof}
In addition to performing operations on an instance of $Q_{e}$ in the string, we can also perform operations on an instance of $\Phi_{x}$ in a line.  To this end we have the following
\begin{lemma}\label{lem:simple-x}
Let $\ell \in \mL_{1}$ with an endpoint $x \in \Lambda$. 

For $(G,M)\in \{(SO(N),\mR^N),(U(N),\mC^N),(SU(N),\mC^N)\}$,
\begin{align*}
-c_0  \phi(\ell)
=
\kappa \sum_{\ell' \in \mE_{\text{ext},x}(\ell) }  \phi(\ell') 
+
\sum_{k=1}^n c_k
\sum_{s \in \mE^k_{x}(\ell)}  \phi(s).
\end{align*}

For $(G,M)\in \{(SO(N),\mS^{N-1}),(U(N),\mS^{2N-1}), (SU(N),\mS^{2N-1})\}$,
\begin{align*}
c_M \phi(\ell)
=
\kappa \sum_{\ell' \in \mE_{\text{ext},x}(\ell) }  \phi(\ell') 
-\frac\kappa{2} \sum_{s\in \mE_{x}(\ell) }  \phi(s).
\end{align*}
Here $c_M$ is defined in \eqref{e:c_S}, $c_0$ is defined in \eqref{e:grad_x}.
\end{lemma}

\begin{proof}
	We prove the case $W_{\ell}=\Phi_{x}^{*}Q_{\ell}\Phi_{y}$
	and the case $W_{\ell}=\Phi_{y}^{*}Q_{\ell}\Phi_{x}$ follows in the same way.
	
	We apply It\^o's formula with respect to the dynamic \eqref{SDE1}  obtained from conditioning on all values of $Q$ and all values of $\Phi$ except at the vertex $x \in \Lambda$

	\begin{align}
		\dif W_{\ell}=\dif \Phi_{x}^{*} Q_{\ell}\Phi_{y} 
		=( \nabla_{x}\mathcal{S}_{2})^{*}Q_{\ell}\Phi_{y}
		+( \nabla_{x}\mathcal{V})^{*}Q_{\ell}\Phi_{y}
		-\1_{\mS}\, c_M W_\ell
		+\dif M_{\ell}
		 			\label{e:Ito-x}
	\end{align}
	for some martingale $M_{\ell}$.
	By \eqref{e:grad_x}, \eqref{e:grad_xS},
	\begin{align*}
		( \nabla_{x}\mathcal{S}_{2})^{*}Q_{\ell}\Phi_{y}
		&=
		\kappa \sum_{e=(x,z) \in E_{\Lambda}}
		\Big[
		(Q_{e}\Phi_{z})^*
		-\1_{\mS}\big(\Re(\Phi_x^* Q_e\Phi_z)\Phi_x\big)^*
		\Big]Q_{\ell}\Phi_{y} 
		\\
		&=
		\kappa \sum_{e=(x,z) \in E_{\Lambda}} 
		 \Big[
		 \Phi_{z}^* Q_{e^{-1}}Q_{\ell}\Phi_{y}
		 - \1_{\mS}
		\Re(\Phi_x^* Q_e\Phi_z)
		 (\Phi_x^* Q_{\ell}\Phi_{y}) \Big] 
		 \\
		 &=
		 \kappa \sum_{e=(x,z) \in E_{\Lambda}} 
		 \Big[
		 W_{e^{-1}\ell}  - \frac12 \1_{\mS} (W_{e}+W_{e^{-1}})W_\ell
		 \Big]
		 \\
		&=\kappa \sum_{\ell' \in \mE_{\text{ext},x}(\ell)}  W_{\ell'}
		-\1_{\mS} \; \frac\kappa{2} \sum_{s \in \mE_{x}(\ell)}  W_s,
	\end{align*}
	where we recall the definition of extension by edges in \eqref{e:extension2}
	and the definition of expansions  by edges in \eqref{e:expansion2e}.
	Moreover, for $M\in \{\mR^{N},\mC^{N}\}$ 
	$$
	( \nabla_{x}\mathcal{V})^{*}Q_{\ell}\Phi_{y}
	=
	\sum_{k=0}^n c_k  |\Phi_x|^k \Phi_x^* Q_{\ell}\Phi_{y}
	=
	c_0 W_\ell+
	\sum_{k=1}^n c_k
	\sum_{\ell' \in \mE^k_{x}(\ell)}  W_{\ell'}
	$$
	where we recall the definition of expansion by null-lines in \eqref{e:expansion-null}. 
	
	 The claim now follows from the law of total expectation, taking expectation first with respect to the conditional law $\mu_{Q,\Phi_{\backslash x}}$ and then  w.r.t. $\nu_x$ from \eqref{mea:dis}, under which $M_{\ell}$ has zero mean, then further averaging the conditioned variables.
\end{proof}

\subsection{General single strings} \label{sec:nonSimple}

Let $\ell$ be a general string, namely, $\ell$ may have repeated edges,
or have two endpoints possibly identical when it is a line.
In these cases new operations (splitting, twisting, gluing) will appear.

We will consider a general situation where an edge $e\in E$ and its inverse $e^{-1}$
appear multiple times in $\ell$, namely, $W_\ell$ has the form of (trace of, if $\ell\in \mL_0$)
$$
(\cdots) Q_{e_1}  (\cdots)  Q_{e_2}  (\cdots)  Q_{e_r} (\cdots) 
$$
where $e_1,\cdots,e_r\in \{e,e^{-1}\}$ 
for $r\in\mN$.


In the result below, we recall from Section~\ref{sec:Remarks} that 
the sets with subscript $e$ take into account 
all the possible locations where $e$ or $e^{-1}$ appear, see e.g. 
\eqref{e:D-union} \eqref{e:S+union}. 

We also define $t_\ell(e)$ to be the number of occurrences of $e$ minus the number of occurrences of $e^{-1}$ 
in the loop $\ell$. 

\begin{lemma}\label{lem:general-e}
	Let $G\in \{SO(N),U(N),SU(N)\}$. 
	Let $\ell\in \mL$ be a string as above
	in which 
	$e$ or $e^{-1}$ appear $r$ times.
	It holds
		\begin{equs}[general-e]
		(-2r\,c_\mfg & +2\nu (r-t_\ell(e)^2))\phi(\ell)
		\\
		=
		&\mp\beta \sum_{\ell' \in \mD^\pm_e(\ell) } \phi(\ell')
		\mp\kappa \sum_{s \in \mB^\pm_e(\ell) } \phi(s)
		\mp 2\mu\sum_{s \in \mT^\pm_e(\ell) } \phi(s)	\mp 2\lambda \sum_{s \in \mS^\pm_e(\ell) } \phi(s)
		\\
		& \mp\beta\nu \sum_{s \in \mE^\pm_{e,\square}(\ell) } \phi(s)
		\mp\kappa \nu \sum_{s \in \mE^\pm_{e,\parallel}(\ell) } \phi(s).
	\end{equs}
Here $c_\mfg$ is as in \eqref{e:c_g} and 
$\lambda,\mu, \nu$ in as in \eqref{e:lambda-mu-nu}.
\end{lemma}

Note that depending on the Lie group, some of $\lambda,\mu, \nu$ above may be $0$. 
Before the proof, we remark that: 

(1) Set $\kappa=0$. By summing \eqref{general-e} over $e$ 
such that 
 $e$ or $e^{-1}$ appear in $\ell$, we recover 
the loop equations in \cite{Cha,SSZloop}, see \cite[Theorem~1.1]{SSZloop} without the merger terms, up to the factor $N$ as discussed in Remark~\ref{rem:no-N} and Remark~\ref{rem:2}.  Note that the sum of $r$ over all $e$ is precisely $|\ell|$ in 
\cite[Theorem~1.1]{SSZloop}.

(2) If $r=1$, we recover Lemma~\ref{lem:simple-e}.

\begin{proof}
We start with $G=SO(N), U(N)$.	Similarly as \eqref{e7}, differentiating $W_\ell$ and substituting the dynamic in Lemma~\ref{lem:DM}, we have 
	\begin{equ}[e:Ito-general]
		\dif W_{\ell} 
		-r \,c_\mfg W_{\ell}\dif t
		= 
		\sum_{i=1}^r
		\mbox{r.h.s. of \eqref{e7} for $e_i$}
		+ \sum_{i=1}^r \sum_{j=i+1}^r 
		\cI_\ell^{ij} \,\dif t
	\end{equ}
	where $\cI_\ell^{ij}$ are the It\^o correction terms, namely, 
	$$
	\cI^{ij}_{\ell}\, \eqdef 
	\begin{cases}
		\Tr \big(Q_{a}\,\dif Q_{e_{i}}\,Q_{b}\,\dif Q_{e_{j}}\,Q_{c}\big),
		& \text{if $\ell \in \mL_{0}$},
		\\[1.5ex] 
		\Phi^*_{u(a)} Q_{a}\,\dif Q_{e_{i}}\,Q_{b}\,\dif Q_{e_{j}}\,Q_{c}\,\Phi_{v(c)},
		& \text{if $\ell  \in \mL_{1}$},
	\end{cases}
	$$
	for some suitable paths $a,b,c$ (which may depend on $i,j$).
	The proof then follows as in \cite{SSZloop},
	with the case $\ell\in \mL_1$ treated mutatis mutandis, 
	which in particular yields the splitting and twisting terms.
	(In fact, writing 
	$ \Phi^*_{u(a)}
	(\cdots)
	\Phi_{v(c)} = \text{Tr}((\cdots)
	\Phi_{v(c)} \otimes \Phi_{u(a)})$, we view $\Phi_{v(c)} \otimes \Phi_{u(a)}$ as a matrix and everything follows  as in \cite{SSZloop}.)
	For completeness and for the reader's convenience, we give the proof for  the case $\ell\in \mL_1$, which also explicitly demonstrates how 
	the operations on lines defined in Sec.~\ref{sec:Operations} show up. 
	
Consider 
the case $e_i=e_j=e\in E$. Namely, $\ell=aebec \in \mL_1$
and $W_{\ell}=  \Phi_{x}^*Q_{a}Q_{e}Q_{b}Q_{e}Q_{c}\Phi_{y}$.
By \eqref{e:martingale1}, 
\begin{equs}[e:CI1]	
\cI_\ell^{ij}
&=
\text{Tr}(Q_{a}\,  \dif \mathcal M_e \,Q_{b} \, \dif \mathcal M_e \, Q_{c}\Phi_{y}  \Phi_{x}^*)/\dif t
\\
&=2\mu \Tr \big( Q_{b^{-1}}  Q_{c}\Phi_{y} \Phi_{x}^*Q_{a}  \big )-2\lambda \Tr\big(Q_e Q_b )\Tr(Q_{a}Q_{e}Q_{c}\Phi_{y} \Phi_{x}^* \big)
\\
&=2\mu W_{ab^{-1}c}-2 \lambda W_{eb}W_{aec}
\end{equs}
where $ab^{-1}c$ is a negative twisting (see \eqref{e:twisting1}) and $(eb, aec)$ is a positive splitting (see \eqref{e:splitting1}).
Consider the case $e_i=e_j^{-1}=e\in E$. 
Namely $\ell=aebe^{-1}c \in \mL_1$.
By \eqref{e:martingale2}, 
\begin{equs}[e:CI2]
\cI_\ell^{ij}
&=
\text{Tr}(Q_{a}\,  \dif \mathcal M_e \,Q_{b} \, \dif \mathcal M_{e^{-1}} \, Q_{c}\Phi_{y}  \Phi_{x}^*)/\dif t
\\
&=-2\mu \Tr (Q_{a}Q_{e}Q_{b^{-1}}Q_{e^{-1}}Q_{c}\Phi_{y} +2\lambda\Tr (Q_{b})\Tr(Q_{a}Q_{c}\Phi_{y}  \Phi_{x}^*)
 \Phi_{x}^*)
\\
&= -2\mu W_{aeb^{-1}e^{-1}c}+2\lambda W_{b}W_{ac}
\end{equs}
where $aeb^{-1}e^{-1}c$ is a positive twisting \eqref{e:twisting2} and $(b,ac)$ is a negative splitting \eqref{e:splitting2}.  

Thus, recalling Remark~\ref{rem:double-counting},
$$
\sum_{i<j} \cI_\ell^{ij}
= \mp\lambda
 \sum_{s \in \mS_e(\ell) } W_s
\mp\mu\sum_{s \in \mT_e(\ell) } W_s
\;.
$$
	The result for $G=SO(N), U(N)$ then follows by 
	taking expectations w.r.t. the conditional law $\mu_{Q_{\backslash e},\Phi}$ and then w.r.t. $\nu_e$ from \eqref{e:LMHY-com} on \eqref{e:Ito-general}, 
 and multiplying $2$ on both sides.
 
The case $G=SU(N)$ introduces an additional drift term arising from the trace components in $\nabla_e\cS_1$ and $\nabla_e\cS_2$. 
 This yields the expansion terms in third line on the right-hand side of equation \eqref{general-e}, 
exactly as in Lemma~\ref{lem:simple-e}. 
Regarding the It\^o correction terms, 
note that in \eqref{e:martingale1} and \eqref{e:martingale1} in Lemma~\ref{lem:DM}
the  terms with coefficients $\nu$ are just the LHS with $\dif \mathcal M_e$ replaced by $Q_e$,
so in other words these are the same string observable $W_\ell$ that we start with.
Namely, we add $2\nu W_\ell$ in \eqref{e:CI1},
 and  $-2\nu W_\ell$ in \eqref{e:CI2};
to sum up these terms, we need to count the number of times $e_i=e_j$
and the number of times $e_i=e_j^{-1}$.
This is counted 
 in \cite{SSZloop} (the end of Step 1 in Proof of Theorem~1.1 therein),
except that here we do not sum over $e$,
and the expression written as $|A(e)|+|B(e)|-(|A(e)|-|B(e)|)^2$ therein
is precisely equal to $r-t_\ell(e)^2$ here. 
\end{proof}

Next, we consider a scenario  where a line $\ell \in \mL_{1}$ has 
both its beginning point and ending point equal to $x\in \Lambda$ (see  Figure~\ref{fig:Gluing-line-into-loop}).
In the result below, we recall from Section~\ref{sec:Remarks} that 
the sets with subscript $x$ take into account 
the operations that happen at both the beginning point and the ending point, see 
\eqref{e:Exp-union}.

\begin{lemma}\label{lem:general-x}
For any line $\ell \in \mL_{1}$ with both beginning and ending points $x \in \Lambda$,
\begin{align*}
	 -2c_0  \phi(\ell) 
	&=\kappa \sum_{\ell' \in \mE_{\text{ext},x}(\ell) }  \phi(\ell') 
	+\sum_{k=1}^n c_k
		\sum_{s \in \mE^k_{x}(\ell)}  \phi(s)
	+2q\sum_{\ell' \in \mG_x(\ell) }  \phi(\ell')\qquad (M\in \{\mR^N,\mC^N\})
\\
	2qN \phi(\ell) 
	&=\kappa \sum_{\ell' \in \mE_{\text{ext},x}(\ell) }  \phi(\ell')
	 -\kappa \sum_{s \in \mE_{x}(\ell) }  \phi(s)
	 +2q\sum_{\ell' \in \mG_x(\ell) }  \phi(\ell')\qquad\quad(M\in \{\mS^{N-1},\mS^{2N-1}\}).
\end{align*}
Here  $q=1$ in the real case and $q=2$ in the complex case as in Lemma~\ref{lem:DMx}.
\end{lemma}

\begin{proof}
	Similarly as \eqref{e:Ito-x}, with $x=y$
	\begin{equs}
		\dif W_{\ell}
		&=\dif \Phi_{x}^{*} Q_{\ell}\Phi_{y} 
		+\Phi_{x}^{*} Q_{\ell} \dif \Phi_{y} 
		+\dif \Phi_{x}^{*} Q_{\ell} \dif\Phi_{y} 
		\\
		&= \mbox{r.h.s. of \eqref{e:Ito-x} for $x$}
		+\mbox{r.h.s. of \eqref{e:Ito-x} for $y$}
		+ \dif \Phi_{x}^{*} Q_{\ell} \dif\Phi_{y}   
	\end{equs}
	where ``r.h.s. of \eqref{e:Ito-x} for $y$'' means r.h.s. of \eqref{e:Ito-x} with exchange of $x,y$ in the obvious way, and
	the last term 
	is the It\^o correction term.
	
	Applying Lemma~\ref{lem:DMx}, the drift terms give us the extensions $\mE_{\text{ext},x}(\ell)$ and expansions by null-lines $ \mE^k_{x}(\ell)$,
	and a term $-2c_M W_\ell$ (in the cases of spheres), exactly as in Lemma~\ref{lem:simple-x}.
	For the It\^o correction term, by \eqref{e:martin-x1},
	\begin{align*}
		\dif \Phi_{x}^{*}Q_{\ell}\dif\Phi_{x }
		=2q \tr(Q_\ell) -\1_{\mS}2 \Phi_x^*Q_\ell \Phi_x
		=2q \sum_{\ell' \in \mG_x(\ell) }  W_{\ell'}
		- \1_{\mS}\, 2 W_\ell\;.
	\end{align*}
	Here we note that $\tr Q_\ell$ is a Wilson loop, which is  gluing of the line $\ell$. 
	It turns a ``closed line'' $\ell$ where the beginning and ending points coincide
	into a loop $\ell$:
	$$
		\mL_1 \ni \ell \mapsto \ell' \in \mL_0.
	$$
	The last  term $- 2 W_\ell$ combined with  the term $-2c_M W_\ell$ above from the drifts leads to the coefficient $2qN$ on the LHS.
\end{proof}

\subsection{Collections of strings} \label{sec:collections}
Unfortunately, the loop equation is not closed within the class of observables associated to single strings.  Indeed, already from considering simple strings, the breaking operation sends a single string to a pair of strings, making it necessary to consider the expectations of products of traces.  Hence, to obtain such a closure, we ought to consider collections of strings.  

To this end, let $s$ be an ordered collection of strings, that is $s=(\ell_{1},\cdots, \ell_{k})$ for some $k \in \N$, where each $\ell_{i}\in \mL$ is a string.   

Fix an edge $e$ where operations will take place. Denote by $r(\ell)$ the total number of instances of $e$ and $e^{-1}$ within the string $\ell$.  
Furthermore, define 
$$
r(s)=\sum_{ i=1}^{k}r(\ell_{i}) \qquad \mbox{and} \qquad  
t(e)\eqdef \sum_{i=1}^kt_{\ell_i}(e)
$$
where as in the last subsection
$t_\ell(e)$ is the number of occurrences of $e$ minus the number of occurrences of $e^{-1}$ 
in the loop $\ell$.

We can now state the main result of the paper.
Recall again that depending on the Lie group, some of the parameters $\lambda,\mu, \nu$ may be $0$.

\begin{theorem}\label{theo:e-final}
	Let $G\in \{SO(N),U(N), SU(N)\}$. Fix an edge $e$ and let $s$ be a collection of strings.  The following identity holds 
	\begin{equs}[e15]
		{}&\Big(-2 \,r(s)\,c_\mfg+2\nu \big(r(s)-t(e)^2\big)\Big) \phi(s)
		\\	=
		&\mp\beta \sum_{s' \in \mD^\pm_e(s) } \phi(s')
		\mp\kappa \sum_{s' \in \mB^\pm_e(s) } \phi(s')\mp 2\mu \sum_{s' \in \mT^\pm_e(s) } \phi(s')
		\mp 2\lambda\sum_{s' \in \mS^\pm_e(s) } \phi(s')
		\\
		&\mp 2\lambda \sum_{s' \in \mM^\pm_{e,U}(s) } \phi(s')
		\mp 2\mu \sum_{s' \in \mM^\pm_{e,\backslash U}(s) } \phi(s')\mp 2\lambda \sum_{s' \in \mX^\pm_{e,U}(s) } \phi(s')
		\mp 2\mu \sum_{s' \in \mX^\pm_{e,\backslash U}(s) } \phi(s') \\
		&\mp	\beta\nu \sum_{s' \in \mE^\pm_{e,\square}(s) } \phi(s')
		\mp\kappa\nu \sum_{s' \in \mE^\pm_{e,\parallel}(s) } \phi(s').
	\end{equs}
	Here $c_\mfg$ is as in \eqref{e:c_g} and 
	$\lambda,\mu, \nu$ in as in \eqref{e:lambda-mu-nu}.
\end{theorem}

\begin{remark}
	It seems plausible that the above form of the master loop equation could be strengthened with some additional work.  For the pure Yang-Mills model, there are two stronger forms of the master loop equation in the literature.  The first is the `unsymmetrized master loop equation' in \cite{Cha,Jafar,OmarRon}, where instead of $|r|$ on the LHS one has $r_{1}$ and on the RHS, the splitting, twisting, and deformation operations are performed only within $\ell_{1}$, while the mergers involve only $\ell_{1}$ and $\tilde{\ell} \neq \ell_{1}$ (rather than all possible pairs $\ell \neq \tilde{\ell} $ as above).  It seems plausible to us that \eqref{e15} could also be strengthened in a similar way using the integration by parts approaches in \cite{Cha,Jafar,OmarRon}.
	
	\medskip
	
	The second variant takes an even stronger form, the so-called `single location master loop equation', was obtained for the pure Yang-Mills model in \cite{CPS2023}.  We have not yet fully investigated this direction, but it is not too difficult to see that extending this to the YMH setting can be reduced to a problem related to the ``word recursion" for the Haar measure (by the same method as in the proof of Theorem 5.7 in \cite{CPS2023}).  Namely, one would need to show that Proposition 6.102 from \cite{CPS2023} extends to string observables, where by conditioning it suffices to consider $\Phi$ as deterministic.  This reduction is a simple consequence of Taylor expanding the exponential in the action to write each YMH observable as a large sum of product Haar measure observables (for which the corresponding loop equation only involves splitting, twisting, and merging), then using the fact that deformation and breaking are actually mergers with a plaquette or a line with a single edge, respectively.
	
	\medskip
	
	Finally, note that although \eqref{e15} is most likely not the strongest form of loop equation satisfied by the invariant measure \eqref{e:LYMH}, our argument applies equally well to the unconditional stochastic dynamic without assuming stationarity in time, the only difference being that the expected value of the time differential should be included and the equation should be stated in what \cite{Cha} refers to as the `symmetric master loop equation', as in \cite{SSZloop}.  
\end{remark}
\begin{proof}
	For each string component $\ell$ of the collection $s$,
	we have seen in Lemma~\ref{lem:general-e} that
	applying It\^o's formula to the conditional SDE \eqref{SDE} gives us
	\begin{equ}[e:one-component]
		\dif W_{\ell}-r(\ell)c_\mfg W_{\ell}\dif t =(\mathcal{D}_{\ell}+I_{\ell})\dif t+\dif M_{\ell}
	\end{equ}
	for a suitable drift term $\mathcal{D}_{\ell}$, It\^{o} correction $I_{\ell}$ (from repeated occurrences of $e$ or $e^{-1}$ within $\ell$), and martingale $M_{\ell}$.  
	Also, as we have seen  in Lemma~\ref{lem:general-e}, for a single string $\ell$,
	after taking expectations and multiplying by $2$, 
	$\mathcal{D}_{\ell}$ leads to the deformation, breaking, and expansion terms,
	while $I_{\ell}$ leads to the twisting and splitting terms.
	
	Now for a collection $s$, 
	we follow the same structure as the corresponding proof in \cite{SSZloop} by applying It\^{o}'s formula to $W_{s}$ and taking expectation w.r.t. the stationary dynamics for the conditional probability $\mu_{Q_{\backslash e},\Phi}$ and then w.r.t. $\nu_e$ from \eqref{mea:dis} on both sides to derive \eqref{e15}. 
	
	Differentiating $W_s$, we have
	$$
	\dif W_s = \sum_{i=1}^k W_{\ell_1} \cdots \dif W_{\ell_i} \cdots W_{\ell_k}
	+ \sum_{i<j} 
	W_{\ell_1} \cdots \dif W_{\ell_i} \cdots  \dif W_{\ell_j}\cdots W_{\ell_k}
	$$
	and for each $ \dif W_{\ell_i}$ in the first term we substitute \eqref{e:one-component}.
	After  taking expectation,
	the first term then gives us 
	all the deformation, breaking, expansion, twisting and splitting terms, and the term 
	$-2\, r(s)\,c_\mfg \phi(s)$ on the LHS.
	
	In comparison to the previous subsections, the only terms which need to be explained are the mergers and switchings, which will be a result of the It\^{o} corrections to the product rule for $\dif W_s $ above, namely
	\begin{equation}\label{prod:cor}
		\frac{1}{2}\sum_{\ell \neq \tilde{\ell}} \E \big (  \dif M_{\ell}\dif M_{\tilde{\ell}}\prod_{\ell' \notin \{\ell, \tilde{\ell}\} }W_{\ell'} \big ).
	\end{equation}
	
	We now compute $ \dif M_{\ell}\dif M_{\tilde{\ell}}$.  The cases to consider are as follows: $\ell$ and $\tilde{\ell}$ are both loops, $\ell$ and $\tilde{\ell}$ are both lines, and finally that one is a line while the other is a loop.  If both are loops, the same calculation as in \cite{SSZloop} leads to mergers of loops.  We now turn to the other two cases.  We first consider $SO(N)$ and $U(N)$ where $\nu=0$.
	
	\medskip
	
	\textbf{Loop/line mergers:} If  $\ell \in \mL_{1}$ and $\tilde{\ell} \in \mL_{0}$, we may write
	$\dif M_{\ell}\dif M_{\tilde{\ell}}$ as a sum of terms like
	\begin{equation}
		\big( \Phi_{x}^{*}Q_{a} \dif \mathcal M_{e_1}Q_{b}\Phi_{y}\big)\,
		\Tr(\dif \mathcal M_{e_2}Q_c ),  \label{e23} 
	\end{equation}
	where $e_1,e_2 \in \{e,e^{-1}\}$ and $\mathcal M$ is as in Lemma~\ref{lem:DM}.  The summation runs over all possible locations where either $e$ or $e^{-1}$ occur.  If $e_1=e_2=e$, then by \eqref{e:martingale3} 
	\begin{equs}
		\eqref{e23}
		&=
		\Tr\big( \dif\mathcal M_e Q_{b}\Phi_{y} \Phi_{x}^{*}Q_{a}\big)\,
		\Tr(\dif\mathcal M_e Q_c )
		\\
		&=
		2\Big(
		\mu \text{Tr} ( Q_{b}\Phi_{y} \Phi_{x}^{*}Q_{a}\, Q_{c^{-1}})-\lambda\text{Tr}(Q_e \, Q_{b}\Phi_{y} \Phi_{x}^{*}Q_{a} \, Q_e\, Q_c) \Big)\dif t	 
		\\
		&=
		\Big(2\mu W_{ac^{-1}b}-2\lambda W_{aeceb} \Big)\,\dif t
		\\
		&=
		\Big(2\mu W_{\ell \ominus \tilde{\ell}}-2\lambda W_{\ell \oplus \tilde{\ell}}\,\Big)\,\dif t
	\end{equs}
	$ac^{-1}b\in \mL_1$ is a negative merger in $\mM^{-}_{e,\backslash U}(s)$,
	see \eqref{e:merger01-1} and $aeceb \in \mL_1$ is a positive merger in $\mM^{+}_{e,U}(s)$.

	If $e_1=e_2^{-1}=e$, then by \eqref{e:martingale4} 
	\begin{equs}
		\eqref{e23}
		&=
		\Tr\big( \dif\mathcal M_e Q_{b}\Phi_{y} \Phi_{x}^{*}Q_{a}\big)\,
		\Tr(\dif\mathcal M_{e^{-1}} Q_c )
		\\
		&=
		-2\Big(\mu \text{Tr} (Q_e\, Q_{b}\Phi_{y} \Phi_{x}^{*}Q_{a}\,Q_e\, Q_{c^{-1}})
		-\lambda\text{Tr}(Q_{b}\Phi_{y} \Phi_{x}^{*}Q_{a} \, Q_c)\Big)\dif t	 
		\\
		&=
		-\Big( 2\mu W_{aec^{-1}eb}-2\lambda W_{acb} \Big)\,\dif t
		\\
		&=
		-\Big(2\mu W_{\ell \oplus \tilde{\ell}}-2\lambda W_{\ell \ominus \tilde{\ell}}\,\Big)\,\dif t
	\end{equs}
	where $aec^{-1}eb\in \mL_1$ is a positive merger in $\mM^{+}_{e,\backslash U}(s)$,
	see \eqref{e:merger01-2} and $acb \in \mL_1$ is a negative merger in $\mM^{-}_{e,U}(s)$.
	
	\medskip
	
	\textbf{Switching:}
	If  both $\ell$ and $\tilde{\ell}$ both belong to $\mL_{1}$, we may write $dM_{\ell}dM_{\tilde{\ell}}$ as a sums of terms
	\begin{equ}[e24]
		\big ( \Phi_{x}^* Q_{a} \dif \mathcal M_{e_1}Q_{b}\Phi_{y} \big ) 
		\big (\Phi_{z}^* Q_{c} \dif \mathcal M_{e_2}Q_{d}\Phi_{w} \big ),
	\end{equ}
	where $e_1,e_2 \in \{e,e^{-1}\}$.  As above, the summation runs over all possible locations where either $e$ or $e^{-1}$ occur.
	
	If $e_1=e_2=e$, then by \eqref{e:martingale3} 
	\begin{equs}
		\eqref{e24}
		&=
		\Tr\big( \dif\mathcal M_e Q_{b}\Phi_{y} \Phi_{x}^{*}Q_{a}\big)\,
		\Tr\big(\dif\mathcal M_e Q_{d}\Phi_{w} \Phi_{z}^{*}Q_{c} \big)
		\\
		&=
		2\Big(
		\mu \text{Tr} \Big( Q_{b}\Phi_{y} \Phi_{x}^{*}Q_{a}\cdot
		Q_{c^{-1}} \Phi_z \Phi_w^* Q_{d^{-1}}
		\Big)-\lambda\text{Tr}\Big(Q_e \cdot Q_{b}\Phi_{y} \Phi_{x}^{*}Q_{a} \cdot 
		Q_e \cdot Q_{d}\Phi_{w} \Phi_{z}^{*}Q_{c} \Big) \Big)\dif t	 
		\\
		&=
		2\Big(\mu
		\big (\Phi_x^* Q_a Q_{c^{-1}} \Phi_z \big )
		\big ( \Phi_w^* Q_{d^{-1}}Q_{b}  \Phi_y \big )-\lambda
		\big ( \Phi_{z}^* Q_{c}Q_{e}Q_{b} \Phi_{y} \big ) 
		\big ( \Phi_{x}^* Q_{a} Q_{e}Q_{d}\Phi_{w} \big )
		\Big)\dif t	
		\\
		&=
		\Big(2\mu W_{ac^{-1}}W_{d^{-1}b}-2\lambda W_{ceb}W_{aed}+ \Big)\,\dif t
	\end{equs}
	where $(ac^{-1}, d^{-1}b)$ is a negative switching in $\mX^{-}_{e,\backslash U}(s)$,
	see \eqref{e:switching1} and $(ceb,aed)$ is a positive switching in  $\mX^{+}_{e,U}(s)$.
	
	If $e_1=e_2^{-1}=e$, then by \eqref{e:martingale4} 
	\begin{equs}
		\eqref{e24}
		&=
		\Tr\big( \dif\mathcal M_e Q_{b}\Phi_{y} \Phi_{x}^{*}Q_{a}\big)\,
		\Tr\big(\dif\mathcal M_{e^{-1}} Q_{d}\Phi_{w} \Phi_{z}^{*}Q_{c} \big)
		\\
		&=
		2\Big(
		\lambda \Tr\Big(
		Q_{b}\Phi_{y} \Phi_{x}^{*}Q_{a}\cdot Q_{d}\Phi_{w} \Phi_{z}^{*}Q_{c}
		\Big)
		-\mu \Tr \Big(
		Q_e \cdot Q_{b}\Phi_{y} \Phi_{x}^{*}Q_{a}\cdot
		Q_e \cdot Q_{c^{-1}} \Phi_z \Phi_w^* Q_{d^{-1}} 
		\Big)\dif t	 
		\\
		&=
		\Big( 2\lambda W_{ad}W_{cb}-2\mu W_{aec^{-1}}W_{d^{-1}eb} \Big)\,\dif t
	\end{equs}
	where $(ad,cb)$ is a negative switching in $\mX^{-}_{e,U}(s)$ and $(aec^{-1}, d^{-1}eb)$ is a positive switching in $\mX^{+}_{e,\backslash U}(s)$, 
	see \eqref{e:switching2}.
	
	Combining the above calculations and applying similar arguments as in Lemma \ref{lem:general-e} we  obtain all the merger and switching terms in \eqref{e15}. 
	
	For $G=SU(N)$ we have extra expansion terms from the drifts as in Lemma \ref{lem:general-e}. 
	Another  difference is that now $\nu=\frac1N\neq 0$.
	However, note that in all the identities \eqref{e:martingale1}--\eqref{e:martingale4},
	the terms with coefficients $\nu$ are just the LHS with $\dif \mathcal M_e$ replaced by $Q_e$,
	so in other words these are the same string observables that we start with and should be put to the LHS 
	of \eqref{e15}.
	By the same discussion as in the end of Proof of Lemma~\ref{lem:general-e},
	we add either $2\nu W_s$ if $e_1=e_2$ or $-2\nu W_s$ if $e_1=e_2^{-1}$ in the above calculations.
	
	To sum up these $\nu$ terms,
	we also follow the same calculation as in \cite{SSZloop} which shows that the contribution from It\^o correction term with coefficient $\nu$ is given by
	\begin{align*}
		-2\nu \Big(\sum_i r(\ell_i)-\sum_i t_{\ell_i}(e)^2
		-2\sum_{i<j}t_{\ell_i}(e)t_{\ell_j}(e)\Big)\phi(s)=-2\nu \big(r(s)-t(e)^2\big)\phi(s).
	\end{align*}
	Here $r(\ell_i)- t_{\ell_i}(e)^2$ is exactly the same as the single string case Lemma~\ref{lem:general-e}.
	The term $t_{\ell_i}(e)t_{\ell_j}(e)$ is from  \cite{SSZloop}(end of Step 2 in Proof  of Theorem 1.1 therein).
	Hence, \eqref{e15} follows.
\end{proof}

Let $s$ be an ordered collection of strings, that is $s=(\ell_{1},\cdots, \ell_{k})$ for some $k \in \N$, where each $\ell_{i}$ is a string.   Fix a vertex $x$ where operations will take place and denote by $r_x(s)$ the total number of strings in $s$ having $x$ as one of the endpoints.
In the result below, again, we recall from Section~\ref{sec:Remarks} that 
the sets with subscript $x$ 
are disjoint unions.

\begin{figure}[h]
	\begin{tikzpicture}
		\node at (-0.2,-0.2) {$x$};
		\draw[thick,colorline,midarrow] (0,0) -- (3,0);
		\node[dot,colorline] at (0,0) {}; \node[dot,colorline] at (3,0) {};
		\node at (1.2,-0.2) {$\ell_1$};
		\draw[thick,colorline,midarrow] (0,0) -- (0,2);
		\node[dot,colorline] at (0,2) {};
		\node at (-0.2,1.2) {$\ell_2$};
		\draw[thick,colorline,midarrow] (0,0) -- (-1,0) -- (-1,-1) -- (0,-1) -- (0,0);
		\node at (-1.2,-0.5) {$\ell_3$};
		\draw[thick,colorloop,midarrow] (1,1) -- (2,1) -- (2,2) -- (1,2) -- (1,1);
		\node at (2.2,1.5) {$\ell_4$};
	\end{tikzpicture}
	\caption{An example of $s$ consisting of three lines $\ell_1,\ell_2,\ell_3$ and a loop $\ell_4$. We have $r_x(s)=4$.} 
\end{figure}
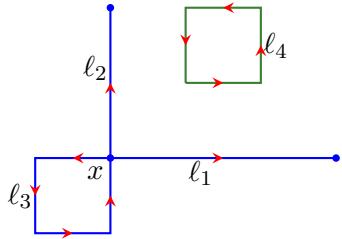

\bt \label{theo:x-final}
Fix a vertex $x$ and let $s$ be a collection of strings. 
For $M\in \{\mR^N,\mC^N\}$,
\begin{align*}
	- &  r_x(s) c_0  \phi(s) \\
	&=
	\kappa \sum_{s' \in \mE_{\text{ext},x}(s) }  \phi(s') 
	+\sum_{k=1}^n c_k
	\sum_{s' \in \mE^k_{\text{exp},x}(s)}  \phi(s')
	+2 q \sum_{s' \in \mG_x(s) }  \phi(s')
	+2 (2-q) \sum_{s' \in \mG_x^{\mR}(s) }  \phi(s')
\end{align*}
and for $M\in \{\mS^{N-1},\mS^{2N-1}\}$,
\begin{align*}
	&[(qN-2)r_x(s)+r_x(s)^2]  \phi(s) 
	\\&=
	\kappa \sum_{s' \in \mE_{\text{ext},x}(s) }  \phi(s') 
	-\kappa \sum_{s' \in \mE_{\text{exp},x}(s) }  \phi(s')
	+2q\sum_{s' \in \mG_x(s) }  \phi(s')
	+2 (2-q) \sum_{s' \in \mG_x^{\mR}(s) }  \phi(s').
\end{align*}
Here as above $q=1$ in the real case and $q=2$ in the complex case. 
\et 

\begin{proof} 
	For each string component $\ell$ of the collection $s$ 
	we have seen in Lemma~\ref{lem:general-x} that 
	applying It\^o's formula to the conditional SDE \eqref{SDE1} gives
	\begin{equation}
		\dif W_{\ell}=(\mathcal{D}_{\ell}+I_{\ell})\dif t+\dif M_{\ell}
	\end{equation}
	for a suitable drift term $\mathcal{D}_{\ell}$, It\^{o} correction $I_{\ell}$ (if the beginning and ending points of $\ell$ coincide), and martingale $M_{\ell}$.  
	Also, as we have seen in 	Lemma~\ref{lem:general-x}, for a single string $\ell$,
	after taking expectation $\mathcal{D}_{\ell}$ leads to the extension and expansion terms (modulo the $c_0$ and $c_M$ terms, precisely treated in Lemma~\ref{lem:general-x}),
	while $I_{\ell}$ results in gluing the coinciding end-points of $\ell$.
	
	Now for a collection $s$, we follow the same structure as the corresponding proof in \cite{SSZloop} by applying It\^{o}'s formula to $W_{s}$ and taking expectation w.r.t. the stationary dynamics for the conditional probability $\mu_{Q,\Phi_{\backslash x}}$ and then w.r.t. $\nu_x$ from \eqref{mea:dis} on both sides to derive the result.

	Similarly  as before,  
	it suffices to
	consider the It\^o corrections that arise from \eqref{prod:cor}. There are four cases.

	(1) $\ell$ begins at $x$ and $\tilde\ell$ ends at $x$.
	Using \eqref{e:martin-x1},
	\begin{equs}
		(\dif \Phi_{x}^{*}Q_{\ell}\Phi_{y} )\,
		(  \Phi_{\tilde y }^{*}Q_{\tilde{\ell} }\dif\Phi_{x })/\dif t
		&=2q \tr (Q_{\ell}\Phi_{y}  \Phi_{\tilde y}^{*}Q_{\tilde{\ell}})
		-  \1_{\mS} \,  2 \Phi_{x}^{*}Q_{\ell}\Phi_{y} \,
		\Phi_{\tilde y }^{*}Q_{\tilde{\ell}}\Phi_{x }
		\\
		& = 2q W_{\tilde\ell \ell} -  \1_{\mS} \, 2W_{\ell}W_{\tilde \ell}
	\end{equs}
	where $\tilde\ell \ell \in \mL_1$ is the gluing of two lines.
	In the second term $(\ell,\tilde\ell)$ is just the original pair of lines.
	
	(2) $\tilde\ell$ begins at $x$ and $\ell$ ends at $x$.
	As in Case (1) this results in 
	$$
	2q W_{\ell \tilde\ell } -  \1_{\mS} \,  2W_{\ell}W_{\tilde \ell}.
	$$
	
	(3) Both $\ell$ and $\tilde\ell$ begin at $x$.
	Using  \eqref{e:martin-x2} with the vector-valued process $R$ there being $Q_{\ell}\Phi_{y}$,
	\begin{equs}
		(\dif \Phi_{x}^{*}Q_{\ell}\Phi_{y})\,
		( \dif  \Phi_{x }^{*}Q_{\tilde{\ell} }\Phi_{\tilde{y} })/\dif t
		&=
		2(2-q) (Q_{\ell}\Phi_{y})^* Q_{\tilde \ell}\Phi_{\tilde y}
		- \1_{\mS} \, 2 (\Phi_{x}^{*}Q_{\ell}\Phi_{y})\,
		( \Phi_{x }^{*}Q_{\tilde{\ell} }\Phi_{\tilde{y} })
		\\
		&= 2(2-q)  W_{\ell^{-1} \tilde\ell}
		-  \1_{\mS} \,  2W_{\ell}W_{\tilde \ell}
	\end{equs} 
	where in the real case
	$\ell^{-1} \tilde\ell\in \mL_1$ is the gluing of two lines,
	and in the complex case the first term is $0$.
	
	(4) Both $\ell$ and $\tilde\ell$ end at $x$, applying \eqref{e:martin-x2}  with $R=Q^{*}_{\tilde{\ell} }\Phi_{x}$ we obtain 
	\begin{align}
		( \Phi_{x}^{*} Q_{\ell}\dif\Phi_{y})\,
		(  \Phi_{x }^{*}Q_{\tilde{\ell} } \dif \Phi_{\tilde{y} })/\dif t&=\Phi_{x}^{*}Q_{\ell} \big (2(2-q)Q^{*}_{\tilde{\ell} }\Phi_{x}-\1_{\mS} 2 \Phi_x (Q^{*}_{\tilde{\ell} }\Phi_{x})^{*} \Phi_x\Big) \nonumber \\
		&=2(2-q)  W_{\ell \tilde\ell^{-1} }
		-  \1_{\mS} \,  2W_{\ell}W_{\tilde \ell}.
	\end{align}
	The proof is finished once we check the coefficient
	$(qN-2)r_x(s)+r_x(s)^2$ on the LHS in the cases of spheres.  Indeed, in all the above cases where gluing happens between two lines, we obtain a term
	$2W_{\ell}W_{\tilde \ell}$, and, 
	recall that in the proof of Lemma~\ref{lem:general-x},
	where gluing happens within one line,
	we also obtain a factor $2$.
	This gives a total coefficient $r_x(s)(r_x(s)-1)$.
	Also, as in the proof of Lemma~\ref{lem:general-x},
	each occurrence of $x$ we obtain
	$c_M=(qN-1)$ from the drift.
	This gives a total coefficient $(qN-1)r_x(s)$.
	Summing up we obtain the result.
\end{proof}


\begin{remark}
One possible generalization of our result  is to
consider more general representations $\sigma: G\to End(\mR^K)$
where $K$ is the dimension of the representation, and the model
\begin{equ}
\mathcal S_{\YMH}(Q,\Phi) 
= 
\beta  \sum_{p\in \CP^+_\Lambda}
 \Re\,\Tr(Q_p)
+\kappa \sum_{e\in E^+_\Lambda }
\Re\, (\Phi_x^* \sigma(Q_e) \Phi_y  )
+ \sum_{z\in \Lambda} V(|\Phi_z|^2)
\end{equ}
where $\Phi$ is $\R^K$ valued.
For example, the adjoint representation of $G$ on $\mfg$ where $K=\dim(G)=\dim(\mfg)$.
Given an open line $l  = e_1 e_2 \cdots e_n$,
it  is  then natural to define  
the Wilson line variable as
$$
W_l := \Phi_{u(e_1)}^* \sigma (Q_{e_1}Q_{e_2}\cdots Q_{e_n})\Phi_{v(e_n)}
=
\Tr_{K\times K}
\Big(\sigma (Q_{e_1}Q_{e_2}\cdots Q_{e_n})\Phi_{v(e_n)}\Phi_{u(e_1)}^* \Big).
$$
Our method should apply to derive the loop equation for this setting.
\end{remark}

\bibliographystyle{alphaabbr}
\bibliography{refs}

\end{document}